\documentclass[a4paper, 11pt]{amsart}
\usepackage[utf8]{inputenc}
\usepackage[english]{babel}
\usepackage{fancyhdr}
\usepackage{amsmath}
\usepackage{amsfonts,amstext,mathrsfs}
\usepackage{amssymb,amsthm,amscd,amsxtra,wasysym,graphicx}
\usepackage{mathrsfs,esint,comment}
\usepackage{tikz-cd}
\usepackage{enumitem,mathtools,dsfont}
\usepackage{palatino}
 \usepackage{mathpazo}
\usepackage{xcolor}

\def\Bibtex{{\rm B\kern-.05em{\sc i\kern-.025em b}\kern-0.08em T\kern-.1667em\lower.7ex\hbox{E}\kern-.125emX}}
\usepackage{hyperref}
\hypersetup{
     unicode=false,          
    pdftoolbar=true,       
    pdfmenubar=true,       
    pdffitwindow=false,     
    pdfstartview={FitH}, 
    pdftitle={Complete K\"ahler manifolds},  
    pdfpagemode=FullScreen,
    pdfauthor={Quang-Tuan Dang},   
    colorlinks=true,  
   linkcolor=purple,         
    citecolor=blue,        
    filecolor=magenta,      
    urlcolor=cyan}

    \usepackage[left=3cm,right=3cm,top=3cm,bottom=2cm]{geometry}

\theoremstyle{plain}    
 \newtheorem{theorem}{Theorem}[section]
 \numberwithin{equation}{section} 
 \numberwithin{figure}{section} 
 \theoremstyle{plain}
 \theoremstyle{plain}    
 \newtheorem{corollary}[theorem]{Corollary} 
 \theoremstyle{plain}    
 \theoremstyle{plain}    
 \newtheorem{lemma}[theorem]{Lemma} 
 \newtheorem{setup}[theorem]{Setup}
 \theoremstyle{remark}
 \newtheorem{remark}[theorem]{Remark}
 \theoremstyle{definition}

\theoremstyle{definition}
\newtheorem{definition}[theorem]{Definition}

\newcommand{\K}{\mathcal{K}}
\newcommand{\X}{\mathcal{X}}
\newcommand{\V}{\mathcal{V}}

\hypersetup{linktocpage=true,
    unicode=false,          
    pdftoolbar=true,       
    pdfmenubar=true,       
    pdffitwindow=false,     
    pdfstartview={FitH}, 
    pdftitle={Monge-Amp\`ere equations},   
    pdfauthor={Quang-Tuan Dang},   
    colorlinks=true,  
   linkcolor=purple,         
    citecolor=blue,        
    filecolor=green,
    urlcolor=blue}

\DeclareMathOperator \Vol {{\rm Vol}}

\DeclareMathOperator \Ric {\textrm{Ric}}
\DeclareMathOperator \ddc{\textrm{dd}^c}
\usepackage[all]{hypcap} 

\title[complete K\"ahler--Einstein metrics]{An iterative construction of  complete K\"ahler--Einstein metrics}

\author{Quang-Tuan Dang and Tat Dat T\^o}
\address{Yau Mathematical Sciences Center, Tsinghua University, Beijing 100084, China}
\email{dangqt@mail.tsinghua.edu.cn $\&$ dangquangtuan10@gmail.com}
 \address{Sorbonne Université, Université Paris Cité, CNRS, IMJ-PRG, F-75005 Paris, France
}
	\email{tat-dat.to@imj-prg.fr}
\thanks{Communicated by K. Ono. Received November 4, 2024. Revised September 17, 2025.}
  \thanks{Q.-T. Dang (corresponding author): Yau Mathematical Sciences Center, Tsinghua University}
    \thanks{T.-D. T\^o: Sorbonne Université, Université Paris Cité, CNRS, IMJ-PRG, F-75005 Paris, France}

\begin{document}

\begin{abstract}
 We extend Tsuji's iterative construction of Kähler--Einstein metrics with negative scalar curvature to non‑compact K\"ahler manifolds with bounded geometry, using Berndtsson's method from the compact setting. Consequently, given a holomorphic surjective map $p:X\to Y$, where $X$ is a weakly pseudoconvex K\"ahler manifold and $Y$ is a complex manifold, and where the smooth fibers admit K\"ahler--Einstein metrics with negative scalar curvature and bounded geometry, we show that the fiberwise K\"ahler--Einstein metrics induce a semipositively curved metric on the relative canonical bundle $K_{X/Y}$. Moreover, our approach also applies to the plurisubharmonic variation of cusp K\"ahler--Einstein metrics.
\end{abstract}

\maketitle

\section{Introduction}

In the 1950s, Calabi conjectured that a compact K\"ahler manifold with negative first Chern class admits a unique K\"ahler--Einstein metric. This was famously shown independently by Aubin~\cite{aubin1978equation} and Yau~\cite{yau1978ricci}, leading to various remarkable results in algebraic geometry for such manifolds. A question arises as to whether K\"ahler--Einstein metrics can be obtained as the limit of a sequence of algebraic metrics induced by the pluri-canonical line bundle; cf.~\cite{Yau87-nonlinear}. Through the Tian--Yau--Zelditch asymptotic expansion of the Bergman kernel,  Donaldson~\cite{Donaldson04-embeddings} iteratively constructed, on a polarized manifold $(X, L)$ with discrete automorphism group, a sequence of "balanced metrics" whose limit is a K\"ahler form with constant scalar curvature. Tsuji~\cite{Tsuji10-construction-KE} introduced a different approach 
to obtaining K\"ahler--Einstein metrics with negative scalar curvature by algebraic approximations. 

Beyond the compact setting, it is natural to investigate the existence of complete K\"ahler--Einstein metrics on non-compact K\"ahler manifolds. Initially, Cheng--Yau~\cite{ChengYau80-complete} established necessary geometric conditions for the existence of complete K\"ahler--Einstein metrics with negative scalar curvature on non-compact manifolds. Later, this result was extended by Kobayashi \cite{Kobayashi84-KE-open} and Tian--Yau~\cite{TianYau87-KE-complete}. At the same time, Mok and Yau~\cite{Mok-Yau83-completeness} showed the existence of a complete K\"ahler--Einstein metric with negative scalar curvature on any bounded pseudoconvex domain. Recently, Wu and Yau~\cite{Wu-Yau20invariant} imposed an assumption on the holomorphic sectional curvature to construct a unique complete K\"ahler--Einstein metric with negative Ricci curvature, using a new complex Monge--Amp\`ere equation. We also refer to ~\cite{Tong-KRF, Huang_Lee_Tam_Tong_2019} for other proofs of Wu--Yau's result using the K\"ahler--Ricci flow. 

Motivated by Tsuji's iteration for numerically approximating canonical K\"ahler metrics in~\cite{Tsuji10-construction-KE,Tsuji13-construction-KE-pseudoconvex}, we study  Tsuji's iterative procedure in the non-compact setting. We aim  to obtain uniform convergence for Tsuji's iterations by making use of the non-compact version of the asymptotic expansion of the Bergman kernels. {In particular, our approach applies to general complex Monge--Amp\`ere equations, rather than being restricted to the K\"ahler--Einstein case as in previous works.}
\medskip

To state our result, we first fix some notation and terminology.  
Let $X$ be a complex manifold of dimension $n$. 
Let $L$ be a positively curved Hermitian holomorphic line bundle on $X$ with reference curvature form $\omega = dd^c \phi_L$, where $\phi_L$ denotes the weight of a Hermitian metric $h_L = e^{-\phi_L}$ on $L$. We consider the space $\mathcal{H}_L$ of (possibly singular) Hermitian metrics on $L$ with positive curvature current; such metrics can be written as $h = e^{-\phi}$ with plurisubharmonic (psh) weight $\phi$, and $\omega_\phi$ denotes the associated curvature current. One may also identify a Hermitian metric $h = e^{-\phi}$ on $L$ with the function $\varphi = \phi - \phi_L$, so that $h = h_L e^{-\varphi}$.   

By $H^0(X, L)$ we denote the space of globally holomorphic sections of $L$. 
Let $\mu$ be a positive measure on $X$. 
For any positive integer $k$, we define a Hermitian product on $H^0(X, L^k)$ via the $L^2$ norm
\begin{equation}
    \|s\|^2_{k\phi,\mu} := \int_X |s|^2 e^{-k\phi} \, {\rm d}\mu.
\end{equation}
{Here the integrand $|s|^2 e^{-k\phi}$ may also be written as $|s|^2_{h_L} e^{-k\varphi}$, where $\varphi = \phi - \phi_L$.}

We then define the {\it Bergman function}
\begin{equation}
    B_{k\phi,\mu}(x) = \sup \{\, |s(x)|^2 e^{-k\phi(x)} \colon s \in H^0(X, L^k),\ \|s\|_{k\phi,\mu} \le 1 \,\}.
\end{equation}
The quantity $\K_{k\phi,\mu} = B_{k\phi,\mu} e^{k\phi}$ is the {\it Bergman kernel} associated with $(L, k\phi, \mu)$. In particular,
\(
\frac{1}{k} \log \K_{k\phi,\mu}
\)
defines a Hermitian metric on $L$.

\medskip
Fixing a smooth measure $\Omega$ on $X$, we are interested in the following complex Monge--Amp\`ere equation
\begin{equation}\label{eq:MA_inf}
    \omega^n_{\phi_\infty} = e^{\phi_\infty - \phi_L}\Omega,
\end{equation}
where 
$\phi_\infty$ is an unknown smooth Hermitian metric on $L$ with positive curvature. Set
$\varphi := \phi_\infty - \phi_L$ and $\omega := dd^c \phi_L$. Then the equation \eqref{eq:MA_inf} can be rewritten as 
\begin{equation}
(\omega + dd^c \varphi)^n = e^{\varphi}\Omega.
\end{equation}
{In particular, if the volume form $\Omega$ satisfies $\Ric(\Omega) = -\omega$, then}
$\Ric(\omega_{\phi_\infty}) = -\omega_{\phi_\infty}$, {so} $\omega_{\phi_\infty}$ is a K\"ahler--Einstein metric with Ricci curvature $-1$.
\medskip

We now describe Tsuji's iteration associated with the equation~\eqref{eq:MA_inf}.  
Assume that $\phi$ is a (possibly singular) Hermitian metric on $L$.  
Let $\K_{k\phi,\mu_\phi}$ be the Bergman kernel associated with the weight $k\phi$ and the measure $\mu_\phi := e^{\phi - \phi_L}\Omega / n!$. The quantity
\begin{equation}\label{eq_beta}
\beta_k(\phi) := \frac{1}{k}\bigl( \log \K_{k\phi,\mu_\phi} - \log d_k \bigr),
\end{equation}
defines again a metric on $L$, with

\[
d_k = \left( \frac{k}{2\pi} \right)^n.
\]


Starting from a metric $\phi_1$ on $L$, we inductively obtain a sequence of Hermitian metrics on $L$ by setting $\phi_{k+1}:= \beta_{k}(\phi_k)$. 
Generalizing the works of Tsuji~\cite{Tsuji10-construction-KE}, Song--Weinkove~\cite{SongWeinkove10-negativecurvature} and Berndtsson \cite{berndtsson09-remarks-tsuji'}, we show that  $\phi_k$ uniformly converges to the metric $\phi_\infty$ which satisfies the equation \eqref{eq:MA_inf}. 

In particular, our result applies to solutions of the general complex Monge--Amp\`ere equation \eqref{eq:MA_inf}, without assuming that the equation arises from K\"ahler--Einstein geometry, in contrast with previous works.


\begin{theorem}\label{thm: main} Let $X$ be an $n$-dimensional K\"ahler manifold equipped with a positively curved line bundle $(L,e^{-\phi_L})$. Let $\Omega$ be a smooth volume form on $X$. 
Assume that there exists a smooth Hermitian metric $\phi_\infty$ on $L$ satisfying the complex Monge--Amp\`ere equation 
\begin{equation}
\omega^n_{\phi_\infty}= e^{\phi_\infty-\phi_L}\Omega.
\end{equation}
Assume moreover that $(X,\omega_{\phi_\infty})$ is a K\"ahler manifold with bounded geometry of order $\ell\geq 5$. 
 Let  $\phi_1$ be a continuous metric on $L$ such that $\sup_X|\phi_1- \phi_{\infty}|\leq C$. Let $\phi_k$ be a sequence of metrics on $L$ defined by Tsuji's iteration as described above.  Then 
 $\sup_X| \phi_k-\phi_\infty|\rightarrow 0$
 as $k\rightarrow \infty$ at the rate $\frac{\log k}{k}$. 
\end{theorem}
{ The hypothesis of bounded geometry is satisfied, for example, by the K\"ahler–Einstein metric on a strongly pseudoconvex domain in  $\mathbb{C}^n$ (cf. \cite{ChengYau80-complete}), or by K\"ahler metrics with cusp singularities along some divisor (cf. \cite{Kobayashi84-KE-open, TianYau87-KE-complete}).}
\begin{remark}
In particular, when $L=K_X$ is an ample line bundle equipped with a smooth Hermitian metric $e^{-\phi_L}$, we can take $\Omega=e^{\phi_{K_X}}$ as a smooth volume form. For any (smooth) metric $e^{-\phi}$ on $K_X$, the $L^2$ norm on $H^0(X,K_X^{(k+1)})$ is defined by   \[\|\sigma\|_{k\phi,\mu_\phi}^2:=c_n\int_X \sigma\wedge\Bar \sigma e^{-(k+1)\phi} e^\phi=\int_X|\sigma|^2e^{-k\phi}.\] Similarly, we define Tsuji's iteration, which is analogous to that in~\cite{berndtsson09-remarks-tsuji'}. The only difference is that the constant  $d_k$ is the dimension of $H^0(X, K_X^{(k+1)})$ in the compact setting. It is known that $\dim H^0(X, L^k)=O(k^n)$ for ample line bundle $L$, so our definition for $d_k$ is the one of Berndtsson modulo a uniform constant. 
\end{remark}

As mentioned above, one considers the special case where $L=K_X$ is ample, it follows from~\cite{ChengYau80-complete, Wu-Yau20invariant} that, under extra conditions, there exists a unique K\"ahler--Einstein metric $\omega_{\rm KE}$  on $X$   such that  $\Ric (\omega_{\rm KE})=-\omega_{\rm KE}$ with
$$\omega_{\rm KE}= \ddc\log (e^u\omega^n)= \ddc\log (\omega^n) +dd^c u$$ for a bounded smooth function $u$. We denote by $\phi_\omega =\log \omega^n$ and $\phi_{\rm KE}=\log (e^u\omega^n) $ as two Hermitian metrics on $K_X$.  Then, we obtain the following Monge--Amp\`ere equation
\begin{equation}\label{eq: KE}
    (\ddc\phi_{\rm KE})^n=e^{\phi_{\rm KE}-\phi_\omega} \omega^n.
\end{equation}
In this case, we define Tsuji's iteration $ \phi_{k+1}=\beta_k(\phi_k)$ with $\beta_k$ defined as in \eqref{eq_beta} and
\begin{equation}\label{eq:d_k_KE}
    d_k=\left(\frac{k }{2\pi}\right)^n  \max\left\{\frac{1}{2},  1+ \frac{{\rm Scal}(\omega_{\rm KE})}{2k}\right\}=\left(\frac{k  }{2\pi}\right)^n\max\left\{\frac{1}{2},  1- \frac{n}{2k}\right\}.
\end{equation}
Then, we obtain the non-compact version of Tusji's iteration (cf. \cite{Tsuji10-construction-KE,SongWeinkove10-negativecurvature, berndtsson09-remarks-tsuji'}). 
\begin{theorem}   
\label{thm: construction-KE}
With assumptions in Theorem~\ref{thm: main}, suppose further that $L=K_X$ with $\phi_\infty=\phi_{\rm KE}$ satisfying the equation~\eqref{eq: KE}.
 Let  $\phi_1$ be a continuous metric on $K_X$ such that $\sup_X|\phi_1- \phi_{\rm KE}|\leq C$. Let $\phi_k$ be a sequence of metrics on $K_X$, defined by Tsuji's iteration as described above.  Then 
 $\sup_X| \phi_k-\phi_{\rm KE}|\rightarrow 0$
 as $k\rightarrow \infty$ at the rate $\frac{1}{k}$. 
\end{theorem}

This result also extends the uniform convergence obtained by the second author~\cite{To_2022} and, more recently, by Yoo~\cite{Yoo_2024}, in the setting of pseudoconvex domains with bounded geometry. 

\medskip
The idea of the proof is to use the following uniform asymptotic expansion of the Bergman function:
\begin{equation}
\label{eq_bergman_asymp_0}\left|B_{k\phi, \omega_\phi^{[n]}} - \left(\frac{k}{2\pi} \right)^n  \left( 1+ \frac{b_1}{2k}\right)\right|_{\mathcal{C}^0} \leq Ck^{n-2}.
\end{equation}
where $\omega^{[n]}:= \frac{\omega^n}{n!}$  and $C>0$ depends only on the bounded geometry of $(X,\omega_\phi)$, extending Berndtsson's work~\cite{berndtsson09-remarks-tsuji'} in the compact setting. A more general asymptotic expansion in $C^k$‑norms was proved by Ma–Marinescu \cite[Problem~6.1]{MaMarinescu-book} (see also \cite[Theorem~6]{Ma_Marinescu_2015} and \cite{Keller-notes-Bergman}). For the reader's convenience, we provide an alternative proof of the uniform Bergman expansion \eqref{eq_bergman_asymp_0} (see Theorem~\ref{thm: Bergman}) using Tian's peak‑section method \cite{Tian90-metrics-algebraic} (see also \cite{Ruan_1998,Lu00-expansion}) together with the Laplace method as in \cite{donaldson-2009-notes,charles15-bergmankernel}. We refer to the works \cite{deMonvel_Sjostrand_75,Tian90-metrics-algebraic,Catlin_1997,Zelditch_1998,Lu00-expansion,Charles_2003,Dai_Liu_Ma_06,MaMarinescu-book,Berman_Berndtsson_Sjostrand_2008,Ma_Marinescu_2015} for further background on the asymptotic behavior of Bergman kernels on compact K\"ahler manifolds.

\medskip
As a consequence, we apply Theorem \ref{thm: construction-KE} to study the variation of complete K\"ahler--Einstein metrics, extending the results in \cite{Tsuji10-construction-KE, Schumacher12-positivity-bundles, Berman_2013} to the non-compact setting. We refer to \cite{paun17-relative-nef, Guenancia_2020, Cao_Guenancia_Paun_2021} for more results in the compact case.

\begin{theorem}\label{thm: psh-variation}
  Let $p:X \to Y$  be a holomorphic surjective map from a weakly pseudoconvex K\"ahler manifold $X$ to a complex manifold $Y$. Let $Y^\circ\subset Y$ be a set of points that are not critical values of $p$ in $Y$ and set $X^\circ=p^{-1}(Y^\circ)$ and $X_y= p^{-1}(\{y\})$. 
    Assume  that 
\begin{itemize}
    \item[(i)] for every $y\in Y^\circ $, $X_y$ admits a K\"ahler--Einstein metric $\omega_{\textrm{KE},y}$ with negative Ricci curvature such that $(X_y,\omega_{\textrm{KE},y})$ has bounded geometry of order $\ell\geq 5$, uniformly in $y$ on compact subsets of $Y^\circ$,

    \item[(ii)] $K_{X/Y}$ admits a singular Hermitian metric $h_{0}=e^{-\phi_{0}}$ with positive curvature current and for every $y_0\in Y^\circ$, there exists a neighborhood $U\Subset  Y^\circ $ of $y_0$ such that  $\|\phi_{0, y}-\phi_{{{\rm KE}, y}}\|_{L^\infty(X_y)}\leq C(U)$ for every $y\in U$.
    
 
\end{itemize}
Then the Hermitian metric on $(K_{X/Y})|_{X^\circ}$ induced by the fiberwise K\"ahler--Einstein metrics is positively curved.  Moreover, if we have the additional condition that
\begin{itemize}
\item[(iii)] the K\"ahler--Einstein potentials $(\phi_{{\rm KE},y})_{y\in Y^\circ}$ are bounded from above near the singular fibers,
\end{itemize}
then this metric extends canonically across $X\setminus X^\circ$ to a positively curved metric on $K_{X/Y}$. 
\end{theorem}
As a consequence, we recover the positivity of the variation of K\"ahler--Einstein metrics on pseudoconvex domains proved in \cite{Tsuji13-construction-KE-pseudoconvex, Choi15-variations-pseudoconvex, Choi_Yoo_2022}, following Tsuji's argument \cite{Tsuji13-construction-KE-pseudoconvex} (cf. Theorem \ref{thm_var_psh_domain}). See also Choi--Yoo~\cite{Choi_Yoo_2021} for a similar result for strongly pseudoconvex domains contained in a K\"ahler manifold, using the Schumacher method~\cite{Schumacher12-positivity-bundles}.

\medskip
Furthermore, the statement of Theorem \ref{thm: psh-variation} also holds in the following setting (cf. Theorem \ref{thm: variation-psh2}).   Let $p:\overline{X} \to Y$ be a holomorphic surjective and proper map between compact K\"ahler manifolds with relative dimension $n$. Let $X:=\overline{X}\setminus D$ where $D = \sum_{i=1}^\ell D_i$ is a reduced divisor with simple normal crossings. Assume that $Y^\circ$ is the largest Zariski open subset of $Y$ such that, if $\overline{X}^\circ := p^{-1}(Y^\circ)$, then every fiber $\overline{X}_y := p^{-1}(\{y\})$ is smooth and, $D_y:=D\cap {\overline{X}_y}$ has simple normal crossings. We denote $X^\circ=\overline{X}^\circ \setminus D $. Then we will show that the current obtained by gluing the fiberwise K\"ahler--Einstein metrics with cusp singularities along $D$ is positive.


Recall that $\omega_y$ is said to have {\em cusp singularities} (or {\em Poincar\'e singularities}) along $D_y$, locally, if $D_y=(z_1\cdots z_\ell=0)$, then $\omega_y$ is quasi-isometric to the cusp
metric: \[\omega_{\rm cusp}:=\sum_{j=1}^\ell\frac{i dz_j\wedge d\Bar z_j}{|z_j|^2\log^2|z_j|^2}+\sum_{k=\ell+1}^n idz_k\wedge d\Bar z_k. \]
Assume that $K_{\overline{X}_y} + D_y$ is ample for every $y \in Y^\circ$. By the well-known result of Kobayashi \cite{Kobayashi84-KE-open} and Tian--Yau~\cite{TianYau87-KE-complete}, building upon Cheng--Yau’s method \cite{ChengYau80-complete}, for each fiber, there exists a unique complete K\"ahler metric $\omega_y$ on $X_y := \overline{X}_y \setminus D_y$ with cusp singularities along $D_y$ and satisfying $\Ric(\omega_y) = -\omega_y$ on $X_y$.  The metric $\omega_y$ extends as a positive current on $\overline{X}_y$; it satisfies \begin{equation}\label{eq: KE-cusp0}
    \Ric(\omega_y)=-\omega_y+[D_{y}]
\end{equation} in the sense of currents on $\overline{X}_y$, where $[D_y]$ is the integration current along $D_y$. We also refer to \cite{bando90-negative-open,Wu08-negative-quasi-proj, Auvray17-Poincare-metrics-divisor, guenancia14-kahler-mixed, Berman-Guenancia14-logcanonical, DiNezza-Lu-cMA-quasi, Guenancia_Wu_2016, Guenancia20-cones-cusps, Biquard-Guenancia-22-degnerating,
dang2023kahler} for related results on complete K\"ahler--Einstein metrics with cusp or mixed singularities.

Over $X^\circ$, the fiberwise K\"ahler--Einstein metrics $\omega_y$ with cusp singularities glue together to define a current $T$. Applying an analogue of Theorem \ref{thm: psh-variation} in this setting (cf. Theorem \ref{thm: variation-psh2}), we obtain the following.  
\begin{theorem}\label{thm:family_cusp} With the setup above, 
   the fiberwise cusp K\"ahler--Einstein metrics $\omega_y$ satisfying \eqref{eq: KE-cusp0} glue to define a current $T$ on $\overline{X}$ such that:
    \begin{enumerate}[label=(\roman*)]
        \item $T$ is positive on $X^\circ = \overline{X}^\circ \setminus D$;
        \item $T$ extends canonically across $(X \setminus X^\circ)\cup D$ to a closed positive current on $\overline{X}$.
    \end{enumerate}
\end{theorem}
We remark that a similar result for the family of twisted conic K\"ahler--Einstein metrics was obtained by Guenancia~\cite{Guenancia_2020}, and a result for the family of cusp K\"ahler--Einstein metrics was proved by Naumann~\cite{Naumann22-positivity} (cf. also \cite[Section 8.1]{Guenancia_2020}). In particular, they used Schumacher’s approach~\cite{Schumacher12-positivity-bundles}, whereas our proof follows Tsuji’s method~\cite{Tsuji10-construction-KE}.  

\subsection*{Organization of the paper} In Section~\ref{sect: background}, we start with preliminary material on (singular) Hermitian metrics on line bundles and local quasi-coordinates. The main Theorems~\ref{thm: main}  and~\ref{thm: construction-KE} are proved in Section~\ref{sect: main thm}, where we also establish an asymptotic expansion for the Bergman kernel associated with weighted spaces. Section~\ref{sect: pos-variation} is devoted to the positivity of the variation of complete K\"ahler--Einstein metrics, where Theorems~\ref{thm: psh-variation} and ~\ref{thm:family_cusp} are proved.
 
\subsection*{Acknowledgement} The authors are grateful to Vincent Guedj and Henri Guenancia for their support and suggestions. We would like to thank Bo Berndtsson for sharing his paper~\cite{berndtsson09-remarks-tsuji'}. We also thank J. Cao, L. Charles, X. Ma, and S. Yoo for their interest in this work and valuable discussions.  The authors would like to thank the referee for useful comments
and suggestions. Part of this work was done while the second author was visiting Vietnam Institute for Advanced Study in Mathematics (VIASM), and he would like to thank VIASM for their hospitality. 
Q. T. Dang is partially supported by the PARAPLUI ANR-20-CE40-0019 project.
T. D. T\^o is partially supported by ANR-21-CE40-0011-01 (research project MARGE), PEPS-JCJC-2024 (CRNS), and Tremplins-2024 (Sorbonne University).

\subsection*{Ethics declarations} The authors declare that they have no conflict of interest.

\section{Notation and background}\label{sect: background}
\subsection{General notation} Throughout the paper,
 $(X,\omega)$ denotes a K\"ahler manifold of dimension $n$ equipped with a K\"ahler metric $\omega$.  We denote by $d=\partial+\Bar{\partial}$ and ${d^c}=\frac{i}{4}(\Bar{\partial}-\partial)$ so that $\ddc=\frac{i}{2}\partial\Bar{\partial}$.
 

Let $L$ be a Hermitian holomorphic line bundle on $X$. The (singular) Hermitian metric on $L$ will be written as $h=e^{-\phi}$, and we will refer to the function $\phi$ as a {\em weight} on $L$. In practice, $\phi$ is understood as a collection of locally integrable functions $\phi^U$ on trivializing open sets $U$, called {\em local weights}. If $e^U$ is a local holomorphic frame of $L$ over $U$, then $|e^U(z)|_\phi^2 := e^{-\phi^U(z)}$.  We will use the weight $\phi$ instead of $h=e^{-\phi}$ for a Hermitian metric on $L$.
The set of weights on $L$ forms an affine space modeled on $L^1_{\rm loc}(X)$. Namely, if $\phi_1$, $\phi_2$ are two weights on $L$, their difference $\phi_1-\phi_2$ is a locally integrable function on $X$.

 Equivalently, a weight may be viewed as a function on the total space of the dual bundle $L^*$ satisfying the log‑homogeneity property $\phi(\lambda v)=\log|\lambda|+\phi(v)$ for all non-zero $v\in L^*$, $\lambda\in\mathbb{C}$. 
 A section $s \in H^0(X,L)$ induces a weight on $L$, denoted $\log|s|$, defined by $\log|s|(v) := \log |\langle s, v\rangle|$.
The pointwise norm of $s$ with respect to the weight $\phi$ is \[ |s|^2_\phi:=|s|^2e^{-\phi}.\]
If $\phi$ is in the class $\mathcal{C}^2$, i.e., it has continuous derivatives of order two the curvature form of the metric $\phi$ is the global form on $X$, locally expressed as $\ddc\phi$ which represents the first Chern class $c_1(L)$ in real de Rham cohomology.  One should note that the notation $\ddc\phi$ is only symbolic: the form is not exact in general. The curvature form of a smooth metric is said to be {\em positive} if the local Hermitian matrix $(\phi_{i\bar j}) = \left(\frac{\partial^2 \phi}{\partial z_i \partial \bar z_j}\right)$ is positive definite.

In general, the curvature $\ddc\phi$ is still well defined as a $(1,1)$‑current, since each local weight $\phi^U$ is only required to be locally integrable. The curvature current of a singular metric $e^{-\phi}$ is said to be {\em positive} when the local weights $\phi^U$ are plurisubharmonic (psh).  We use the notations $\omega_\phi$ and $\ddc\phi$ interchangeably for the curvature form/current of the weight $\phi$.

Given a positive measure $\mu$ on $X$ and a weight $\phi$ on $L$, we obtain an $L^2$ norm on $H^0(X,L^k)$ defined by
\[ \|s\|^2_{k\phi,\mu}:=\int_X|s|^2e^{-k\phi} \,{\rm d}\mu.\]
The {\em Bergman kernel of order $k$} associated with $(L,\phi,\mu)$ is defined for any $x\in X$ by 
\[ \K_{k\phi,\mu}(x):=\sup\{|s(x)|^2: s\in H^0(X,L^k) , \|s\|_{k\phi,\mu}\leq 1\}.\]
Throughout the paper, the symbol $C$ denotes a positive constant whose value may vary from line to line.

\subsection{Bounded geometry on non-compact K\"ahler manifolds}
We recall the definition of bounded geometry for  K\"ahler manifolds (cf. \cite{Kobayashi84-KE-open,ChengYau80-complete,Wu-Yau20invariant}). Denote by $B_{\mathbb{C}^n}(0,r)$ the open ball centered at the origin $0\in\mathbb{C}^n$ of radius $r>0$ with respect to the Euclidean metric $\omega_{\mathbb{C}^n}$.
\begin{definition}\label{def:bd_geo}
Suppose $(X,\omega)$ is a $n$-dimensional K\"ahler manifold. We say that $(X,\omega)$ has {\em bounded geometry} of order $k\in\mathbb{N}$ if there exists $0<r$ such that for any $p\in X$ there is  a domain $U$ in $\mathbb{C}^n$ and a  holomorphic map   $\psi:U  \rightarrow X $ of maximal rank everywhere,  satisfying 
\begin{enumerate}
    \item $B_{\mathbb{C}^n}(0,r)\subset U $ and $\psi(0)=p;$
    
    \item  On $U$, there exists a constant $c>0$ depending only on $n$, $r$ such that $$c^{-1}\omega_{\mathbb{C}^n}\leq \psi^*\omega \leq c\omega_{\mathbb{C}^n };$$
    \item  there is a constant $A_k>0$ depending only on $k$ such that
$$\sup_{x\in B_{\mathbb{C}^n}(0,r)}\left| \frac{\partial^{|\alpha|+|\beta|}g_{i\bar j}(x)}{\partial z^\alpha \partial \bar z^\beta}\right|\leq A_k, \;\;\forall\;|\alpha|+|\beta|\leq k,$$
\end{enumerate}
where $g_{i\bar j}$ is the component of $\psi^*\omega$ on $U$ in terms of natural coordinates $(z^1, \ldots ,z^n)$, and $\alpha$, $\beta$ are the multiple indices with $|\alpha|=\alpha_1+\cdots+\alpha_n$.

The map $\psi$ is called a {\em quasi-coordinate map} and the pair $(U,\psi)$ is called a {\em quasi-coordinate chart} of $X$. 
\end{definition}
We have the following normal coordinates for a K\"ahler manifold of bounded geometry.
\begin{lemma}\label{lem:normal_chart}
Suppose $(X,\omega)$ is a K\"ahler manifold with bounded geometry of order at least $ 5$.  Then there exists   $c>0$, $A_k>0$, for $k=1,2, \ldots $  such that for  any point $p\in X$,  there exist $0<\varepsilon=\varepsilon(p)< r$,  and a holomorphic map  $\kappa:  B_{\mathbb{C}^n}(0,\varepsilon) \rightarrow X $ which is biholomorphic on its image  with $\kappa(0)=p$, $\kappa^* \omega=dd^c \varphi$ satisfying the conditions (2), (3) in  Definition~\ref{def:bd_geo}, and $\varphi_{a\overline{b}}(0)=\delta_{a b}$, 
$\varphi_{a\overline{b} c}(0)=\varphi_{a\overline{b}\overline{c}}(0)=0$, and  $\varphi_{a \overline{b}\overline{c} \overline{d}}(0)=\varphi_{a b c\overline {d}}(0)=0$. 
\end{lemma}
\begin{proof} 
For any $p\in X$, we fix a local quasi-coordinate chart $(U,\psi)$ at $p$  as in Definition~\ref{def:bd_geo}.  By complex linear transformation we can assume that $\psi^*\omega =dd^c \phi$ on $U$ with $\phi_{i \bar j}(0)=\delta_{ij}$.  Since $\psi$ is of maximal rank everywhere, there exists $0<\varepsilon_1=\varepsilon_1(p)<r$ such that $\psi: B(0, \varepsilon_1)\rightarrow 
\psi(B(0, \varepsilon_1))\subset X$ is biholomorphic.  We can choose $\varepsilon_2<\varepsilon_1$ depending on $A^i_{k\ell}, B^i_{mnp} $  such that the holomorphic map  $\Tilde{\psi}: B(0,\varepsilon_2) \rightarrow B(0,\varepsilon_1) $,  $z^i= \Tilde\psi(w)^i= w^i+ A^i_{k\ell}w^kw^\ell + B^i_{mnp} w^m w^n w^p$   is biholomorphic on its image, where  the coefficients $A^i_{k\ell},B^i_{mnp} $  will be chosen hereafter. 

\medskip
Now we have  $\kappa=\psi\circ \tilde\psi:B(0,\varepsilon_2)\rightarrow X$ is biholomorphic on its image,  and the pull-back metric $ \kappa^*\omega =dd^c \varphi$  satisfies
\begin{align*}
   \varphi_{a\bar b} (w)&=\phi_{a\bar b}(\tilde \psi(w)) + 2A^i_{a k}\phi_{i\bar b}w^{k}  + 2\overline{A^j_{b\ell}\phi_{a\bar j}}\overline{w}^\ell + 3B^i_{amn}w^m w^n  \phi_{i\bar b} + 3 \overline{B^j_{bpq}}\phi_{a\bar j}\overline{w}^p\overline{w}^q \\
   &\quad + 4 A^{i}_{a k}\overline{A^j_{b\ell}}\phi_{i\bar j}w^k\overline{w}^\ell+O(|w|^3). 
\end{align*}
Therefore we have
$$
\varphi_{a\bar b c}(0) = \phi_{a\bar b c}(0)+ 2 A^{b}_{ac};  \varphi_{a\bar b \bar c}(0) = \phi_{a\bar b \bar c}+ 2 \overline{ A^{b}_{ac}}; \varphi_{a\bar b cd}(0) = \phi_{a\bar b c d}(0) +6 B^{b}_{acd}+C,
$$
where $C$ only depends on $A^{b}_{ac}$ and $\phi_{a\bar b c}(0)$.
By choosing $A^b_{ac}=-\frac{1}{2}\phi_{a\bar b c}(0)$ and then $B^b_{acd}=-\frac{1}{6}\phi_{a\bar b c d}(0) - C$,  we obtain the local normal coordinates for any $p\in X$. In particular, $A^{a}_{bc}$ and $B^a_{bcd}$ are uniformly bounded for any $p\in X$ by the  definition of bounded geometry. Since $\varepsilon_2<\varepsilon_1$ depends only on $A^{a}_{bc}$ and $B^a_{bcd}$,  we get the map $\kappa: B(0,\varepsilon_2)\rightarrow X$ as desired.
\end{proof}

\section{Uniform convergence of Tsuji’s iteration}\label{sect: main thm}
\subsection{Asymptotic expansion of  Bergman kernels}
The asymptotic expansion of Bergman kernels (Theorem \ref{thm: Bergman} below) in $C^k$-norms  for Kähler manifolds with {\it bounded geometry} was established in 
Ma–Marinescu \cite[Problem 6.1]{MaMarinescu-book} (cf. also  \cite[Theorem 6]{Ma_Marinescu_2015}.  In this section,  we provide an alternative proof for the asymptotic expansion of Bergman kernels in $\mathcal{C}^0$-norm, which will be used later. Our proof is inspired by Donaldson \cite{donaldson-2009-notes} and  Charles~\cite{charles15-bergmankernel} by using the Laplace approximation. We refer to \cite[Section 6]{MaMarinescu-book} for general results on the asymptotic expansion of the Bergman kernel on complete K\"ahler manifolds.

\begin{theorem}\label{thm: Bergman} Let $(X, \omega_\phi)$ be a complete K\"ahler manifold of dimension $n$, where  $\omega_\phi:=dd^c\phi>0$ is the curvature of a  Hermitian metric $h=e^{-\phi}$ on a line bundle  $L$. 
Suppose that $(X,\omega_\phi)$ has bounded geometry of order $\ell\geq 5$. Then
$$\left|B_{k\phi, \omega_\phi^{[n]}} - \left(\frac{k}{2\pi} \right)^n  \left( 1+ \frac{b_1}{2k}\right)\right|_{\mathcal{C}^0} \leq Ck^{n-2}$$
where $\omega^{[n]}:=\omega^n/n!$, $b_1$ is the scalar curvature of $\omega_\phi$,   $C>0$ depending on the bounded geometry of $(X, \omega_\phi)$, and $B_{k\phi, \omega_\phi^{[n]}}=\K_{k\phi, \omega_\phi^{[n]}}(x)e^{-k\phi(x)} $ denotes the Bergman function associated to $(k\phi,\omega_\phi^{[n]})$.
\end{theorem}
We will need the following Laplace approximation.
\begin{lemma} \cite[Theorem 7.7.5]{Hormander-book-Analysis1}  \label{lem:Laplace}
Let $K\subset \mathbb{R}^{m}$
be a compact set, $U$ an open neighborhood of $K$ and $k$ a positive integer. If $\lambda>0$,  $u\in C^{2k}(K)$ and $f\in C^{3k+1}(U)$, $ f\leq 0$ in $U$, $f'(x_0)=0$, $f'(x)\neq 0$ in $K\backslash\{x_0\}$, $\det H_f(x_0)\neq 0
$ then 
\begin{eqnarray*}
\left|\int u(x) e^{\lambda f(x)}\, {\rm d}x - \frac{\left(\frac{2\pi }{\lambda}\right)^{\frac{m}{2}}e^{\lambda f(x_0)}}{(\det H_f(x_0))^{\frac{1}{2}}}   \sum_{j<k} \lambda^{-j} L_j u\right|\leq  C\lambda^{-k}\sum_{|\alpha|\leq 2k}\sup |D^\alpha u|
\end{eqnarray*}
for a constant $C$ depending on $\|f\|_{C^{3k+1}(U)}$, 
where $H_f$ denotes the Hessian matrix of $f$,
$$L_j u = \sum_{s-r=j} \sum_{2s\geq 3r} 2^{-s}  \left\langle H_f(x_0)^{-1}D,D\right \rangle^s (g^r_{x_0} u)(x_0)/r!s!,$$
with
$$ g_{x_0}(x)=f(x)-f(x_0) - \frac{1}{2}  \left\langle  H_f(x_0)(x-x_0), x-x_0 \right \rangle.$$ 
In particular, $L_0 u= u(x_0)$ and 
\begin{eqnarray*}
L_1 u(x_0) &=&\frac{1}{2} \Bigg[
 u \left\{ - f_{ik\ell}f_{jrs} \left(\frac{1}{4}f^{ij}f^{k\ell}f^{rs}   
+ \frac{1}{6} f^{ij}f^{ks}f^{r\ell} \right) +\frac{1}{4} f^{ij}f^{kl}f_{ijk\ell} \right\}\\
&& \quad + f^{sq} f^{rp}f_{srq} u_p 
- {\rm Tr}(H_u H_f^{-1}) \Bigg]_{x=x_0}.
\end{eqnarray*}
\end{lemma}

\begin{proof}[Proof of Theorem \ref{thm: Bergman}]
We will establish lower and upper bounds for the Bergman kernel by adapting the method of peak sections for compact K\"ahler manifolds, as developed in \cite{Tian90-metrics-algebraic} (see also \cite{Ruan_1998,Lu00-expansion, SongWeinkove10-negativecurvature}), and by using the Laplace method as  in \cite{donaldson-2009-notes,charles15-bergmankernel}.
\medskip

For any $x_0\in X$, we take  $(U, z=(z^1,\ldots,z^n))$ a normal coordinate chart as in Lemma~\ref{lem:normal_chart} centered at $x_0$. 
Then $L$ has a local holomorphic frame $g$ defined on $U$ on which the function $\psi=\phi -\log |g|^2 = -\log|g|^2_{\phi}$  has minimum at $x_0$  and 
\begin{equation}\label{eq_lc}
    \psi(z)=  \sum_{j=1}^n |z^j|^2+\sum_{1\leq i,j,m,\ell\leq n}{\frac{1}{2!2!}}\phi_{i\bar j m\bar \ell} z^i \bar z^{ j} z^{m}\bar z^{ \ell}+O(|z|^5).
\end{equation}
Here for example, we  can choose $a_i, b_{ji}, c_{\ell m i}, d_{pqrs}$ such that $\log |g|^2 (z)= \phi(0)+ a_i z^i+ b_{ji} z^jz^i +c_{\ell m i}z^mz^\ell z^i + d_{pqrs}z^pz^qz^rz^s$ and use Taylor's expansion  to have \eqref{eq_lc}. 

\medskip
Let $\eta: [0,\infty)\to [0,1]$ be a cut-off function satisfying $\eta(t) = 1$ for $t\leq 1/2$, $\eta(t)= 0$ for $t\geq 1$, $ |\eta'(t)|\leq 4$ and $ |\eta''(t)|\leq 8$. Define a $L^k$-valued $(1,0)$ form on $X$:
$$\alpha=\bar \partial \left[\eta\left(\frac{k|z|^2}{(\log k)^2}\right)\right] g^{k},$$
which vanishes outside $A_k:=\{z| (\log k)^2/(2k)\leq |z|^2\leq (\log k)^2/k\}\subset U$. Let $dV$ denote the Lebesgue measure on the local chart $(U,z)$. We define a weight function $\chi$, supported in $U$, by setting \[\chi=(n+2)\eta\left(\frac{k|z|^2}{2(\log k)^2}\right)\log\left(\frac{k|z|^2}{2(\log k)^2}\right) \]
for $z\in U$, and $\chi=0$ outside $U$.
We observe that $\chi\in \ [-2(n+2), 0]$ on $A_k$  and can be approximated by a decreasing sequence of smooth functions $\{\chi_m\}_{m\in\mathbb{N}}$; cf.~\cite[pages 104--105]{Tian90-metrics-algebraic}.
It follows from the bounded geometry property that for $k>0$ sufficiently large, 
$$\textrm{Ric}(\omega_\phi)+k\omega_\phi+\ddc\chi_m \geq \frac{k}{C} \omega_{\phi},$$
for a uniform constant $C>0$. 
On the other hand, there exists a constant $C>0$ depending on $n$ and the bounded geometry of $(X,\omega_\phi)$,
\begin{equation}\label{eq: alpha-norm}
    e^{-\chi}|\alpha|^2_{\omega_\phi} \leq  C\left|\eta'\left(\frac{k|z|^2}{(\log k)^2}\right)\right|^2 \frac{k^2 }{(\log k)^4} \phi^{i\bar j}z^i\bar z^j |g|^{2k}\leq C \frac{k |g|^{2k}}{(\log k)^2} ,
\end{equation}   on $A_k$, otherwise $\alpha=0$. Since on $A_k$, we have 
$e^{-k\psi} = (1- \sum_{j=1}^n  |z ^j|^2)^k + O(|z|^{3})$ as follows from the choice of coordinates~\eqref{eq_lc}, hence locally
$$e^{-k\psi}\omega_\phi^n\leq \left(1-\frac{1 }{2}|z|^2\right )^k n!{\rm d}V_{},$$
where ${\rm d}V=\left(\frac{i}{2} \right)^n dz^1\wedge d\Bar{z}^1\wedge\cdots\wedge dz^n\wedge d\Bar{z}^n$.
It follows from~\eqref{eq: alpha-norm} that
\begin{align}    \label{eq_est_alpha}\int_{A_k}|\alpha|^2_{\omega_\phi}e^{-\chi-k\phi}\omega_\phi^{[n]}&\leq C\frac{k}{(\log k)^2}\int_{A_k}|g|^{2k}e^{-k\phi}\omega_\phi^{[n]}\\ \nonumber
    &=C\frac{k}{(\log k)^2}\int_{A_k}e^{-k\psi}\omega_\phi^{[n]}\\ \nonumber
    &\leq C\frac{k}{(\log k)^2}\left(\frac{(\log k)^2}{k} \right)^n \left( 1-\frac{(\log k)^2}{k}\right)^k\\ \nonumber
    &\leq C(\log k)^{2n-2}k^{1-n-(\log k)/2}<\infty \nonumber
\end{align} 
because $(1-(\log k)^2/k)^k=e^{k\log(1-(\log k)^2/k)}\leq e^{-(\log k)^2/2}=k^{-(\log k)/2}$, where $C>0$ only depends on $n$ and the bounded geometry of $(X,\omega_\phi)$.

We apply  H\"omander's $L^2$ estimate (see~\cite[Chapter 4]{Hormander-book-complex} or~\cite[Proposition 1.1]{Tian90-metrics-algebraic}) with the weight $\chi$: there exists a smooth section 
$u$ of $L^k$ on $X$ such that $\bar \partial u= \alpha$  and 
\begin{eqnarray*}
 \int_X|u|^2e^{-k\phi-\chi}\omega_\phi^{[n]}
&\leq &\frac{C}{k}\int_X|\alpha|^2_{\omega_\phi} e^{-k\phi-\chi}\omega_\phi^{[n]}\\
&= &\frac{C}{k}\int_{A_k}|\alpha|^2_{\omega_\phi} e^{-k\phi-\chi}\omega_\phi^{[n]}\\
&\leq &C (\log k)^{2n-2}k^{-\log k/2-n },
\end{eqnarray*}
where the last inequality follows from \eqref{eq_est_alpha}.
Because of the non-integrability of $e^{-\chi}$ at $x_0$, one must have $u(x_0)=0$.

\smallskip
Since  $\chi\leq 0$ we have
$$\int_X | u |^2e^{-k \phi}\omega_\phi^{[n]} \leq \int_X|u|^2e^{-k\phi-\chi}\omega_\phi^{[n]} \leq C (\log k)^{2n-2}k^{-\log k/2-n }.$$
It follows that  $f:=\eta g^k -u$ is a holomorphic section of $L^k$,  
 satisfying $|f(x_0)|^2=e^{k\phi(0)}$ since $u(x_0)=0$  and 
\begin{eqnarray*}
\| f\|^2_{k\phi}=\int_X |f|^2 e^{-k\phi}\omega_\phi^{[n]}= \int_{U_k}|g|^{2k} e^{-k\phi}\omega_\phi^{[n]}+ O(k^{-n-2}),
\end{eqnarray*}
where  $U_k=\{z: |z|^2\leq (\log k)^2/k\}$.  On $U_k$ we have $$\omega_\phi^{[n]} = (1+ \phi_{i  \bar j\ell\bar\ell} z^i\bar z^j +o(|z|^3)){\rm d}V.$$
Applying the Laplace approximation (Lemma \ref{lem:Laplace}) with $u= 1+ \phi_{i  \bar j\ell\bar\ell} z^i\bar z^j +o(|z|^3)$ and $f=-\psi$ where $\psi = -\log|g|^2_\phi$,  we have
\begin{eqnarray*}
\int_{U_k} |g|^{2k} e^{-k\phi}\omega_\phi^{[n]} &=&  \int_{U_k}  e^{-k\psi}\omega_\phi^{[n]}  \\
&=& \int_{U_k} e^{-k \psi}  (1+ \phi_{i  \bar j\ell\bar\ell} z^i\bar z^j +o(|z|^3))\,{\rm d}V \\
&=&  \left(\frac{2\pi}{k}\right)^n \left( 1+ \frac{1}{2k} \phi_{i\bar i j\bar j }  + O(k^{-2}) \right)\\
&=&\left(\frac{2\pi}{k}\right)^n \left( 1- \frac{1}{2k} S_{\omega_\phi}(0)  + O(k^{-2}) \right),
\end{eqnarray*}
where $O\left({k^{-2}}\right) $ dominated by $C/k^{-2}$  on $U$ with a uniform constant $C$ depending only on the bound of $|D^\alpha\phi|$ on $U$ with $|\alpha|\leq 7$.  Therefore
\begin{equation*}
\K_{k\phi,\omega_\phi^{[n]}}(x_0)\geq \frac{ |f(x_0)|^2}{\|f\|^2_{k\phi}}\geq  \left(\frac{k}{2\pi}\right)^n e^{k\phi(0)}\left( 1+ \frac{1}{2k} S_{\omega_\phi}(0)  + O(k^{-2}) \right),
\end{equation*}
hence we get a lower bound for $\K_{k\phi,\omega_\phi^{[n]}}e^{-k\phi}$.

\medskip
We look for an upper bound for $\K_{k\phi,\omega_\phi^{[n]}}e^{-k\phi}$ following the strategy in~\cite{charles15-bergmankernel}.  It suffices to work on the coordinate chart $(U,z)$ as above. 
Let $f$ be a local holomorphic section of $L^k$ on $U$. 
We claim that
\begin{equation*}
|f(0)/g^k(0)|^2\leq \frac{\int_{\Delta_R} |f/g^{k}|^2e^{-k\psi_0 }  \rho \,{\rm d}V}{\int_{\Delta_R}e^{-k\psi_0}\rho\, {\rm d}V},
\end{equation*}
where $\Delta_R\subset U$ is a polydisc with radius $R$ and
$$\psi_0= |z^j|^2+\phi_{i\bar j m\bar \ell} z^i \bar z^{j} z^{m}\bar z^{ \ell},\; \rho=1+ \phi_{ i \bar j \ell\bar \ell}z^i\bar z^j.$$ 
{Indeed, using polar coordinates, one knows that $\int_{\Delta_R}z^m\Bar z^\ell dV=0$ if $m\neq \ell$. Hence, for any multi-index $\alpha\in\mathbb{N}^n$ we have that \[\int_{\Delta_R}e^{-k\psi_0}z^\alpha  \rho \,{\rm d}V=0.\] It follows that for any holomorphic function $h: U\to\mathbb{C}$,
\[ \int_{\Delta_R}|h|^2 e^{-k\psi_0}\rho \,{\rm d}V=\int_{\Delta_R}\big(|h(0)|^2+|h-h(0)|^2 \big) e^{-k\psi_0}\rho \,{\rm d}V
. \] We thus apply with the holomorphic function $f/g^k$ (since $g$ has no zeros in the polydisc $\Delta_R$ of radius $R$ for $R>0$ sufficiently small) to obtain our claim.}

Choosing $R=k^{-2/3}$,
we have $$-k\psi_0\leq -k\psi + Ck^{-7/3} \;
\text{and}\; 
\rho dV\leq (1+Ck^{-2} )\omega_\phi^{[n]} $$
hence
\begin{align*}
\int_{\Delta_R} |f/g^k|^2e^{-k\psi_0 } \rho dV&\leq  (1+ Ck^{-2}) \int _{\Delta_R}  |f/g^k|^2e^{-k\psi}\omega_\phi^{[n]}\\
&=(1+ Ck^{-2}) \int _{\Delta_R}  |f|^2e^{-k\phi} \omega_\phi^{[n]}.
\end{align*}
Using the Laplace approximation again, we have 
\begin{equation*}
\int_{\Delta_R}e^{-k\psi_0}\rho dV= \left(\frac{2\pi}{k}\right)^n\left( 1- \frac{1}{2k} S_{\omega_\phi}(0)  + O(k^{-2}) \right).
\end{equation*}
Therefore we get
\begin{align*}
|f(0)/g^k(0)|^2&\leq \frac{\int_{\Delta_R} |f/g^k|^2e^{-k\psi_0 }\rho dV}{\int_{\Delta_R}e^{-k\psi_0}\rho dV} \\ \nonumber
&=\left(\frac{k}{2\pi}\right)^n \left( 1+ \frac{1}{2k} S_{\omega_\phi}(0)  + O( k^{-2}) \right)\int _{\Delta_R}  |f|^2e^{-k\phi} \omega_\phi^{[n]}.
\end{align*}
Since $|g^k(0)|^2=e^{k\phi(0)}$, we infer that
$$ |f(0)|^2\leq e^{k\phi(0)}\left(\frac{k}{2\pi}\right)^n \left( 1+ \frac{1}{2k} S_{\omega_\phi}(0)  + O( k^{-2}) \right)\int _{\Delta_R}  |f|^2e^{-k\phi} \omega_\phi^{[n]},$$
hence we get a upper bound for $\K_{k\phi,\omega_\phi^{[n]}}e^{-k\phi}$.
\end{proof}

\subsection{ Convergence of Tsuji’s iteration}
 Recall that $L\to X$ is a positively curved line bundle on $X$ equipped with a smooth Hermitian metric $\phi_L$. Let $\Omega$ be a smooth volume form on $X$. For any Hermitian metric $\phi$ on $L$, we denote by $\mu_\phi=\frac{1}{n!}e^{\phi-\phi_L}\Omega$, which is a well-defined positive measure on $X$. We prove the following lemma,  which is a variant version of \cite[Lemma 3]{berndtsson09-remarks-tsuji'}.
\begin{lemma}
\label{lem_key}
Let $\phi_\infty$ be the solution of
$$ (dd^c\phi_{\infty})^n=e^{\phi_{\infty}-\phi_L}\Omega,\quad\omega_{\phi_{\infty}}=\ddc\phi_\infty>0$$ such that $(X,\omega_{\phi_\infty})$ has bounded geometry of order $\ell\geq 5$.
Suppose that  $C_1$, $C_2$ are two real numbers  satisfying $C_1\leq \phi -\phi_{\infty}\leq C_2 $. 
Then   we  have
\begin{equation}
\label{est_1}
\beta_k(\phi)-\phi_{\infty}\geq  \frac{k-1}{k}C_1 - \varepsilon_k
\end{equation}
and
\begin{equation}
\label{est_2}
\beta_k(\phi)-\phi_{\infty}\leq   \frac{k-1}{k}C_2+ \varepsilon_k,
\end{equation}
where  $\varepsilon_k= c/k^2$ with $c>0$ only depending on the bounded geometry of $\phi_{\infty}$. Moreover, if $\omega_{\phi_\infty}$ is a K\"ahler--Einstein metric with negative curvature, and  $d_k$ is defined in \eqref{eq:d_k_KE}, then we can obtain $\varepsilon_k=c/k^3$.
\end{lemma}

\begin{proof}
Since $ (dd^c\phi_{\infty})^n=e^{\phi_{\infty} -\phi_L}\Omega=n! \mu_{\phi_\infty}$, using the first order asymptotic of  Bergman kernels for the weight $\phi_{\infty}$ (cf.~Theorem~\ref{thm: Bergman}), we have
\begin{eqnarray*}
\K_{k\phi_{\infty}, \mu_{\phi_\infty}}e^{-k \phi_{\infty}}  &=&\K_{k\phi_{\infty},\omega_{\phi_{\infty}}^{[n]}}e^{-k \phi_{\infty}} \\
&=&\left(\frac{k  }{2\pi}\right)^n \left(1+O\left(\frac{1}{k}\right)\right).
\end{eqnarray*}  
By definition
$d_k= \left(\frac{k  }{2\pi}\right)^n$ we have
\begin{equation}\label{est_kernel}
\beta_k(\phi_{\infty})= \frac{1}{k}\log(\K_{k\phi_{\infty
}}/d_k)= \phi_{\infty} + \varepsilon_k,
\end{equation} 
where  $\varepsilon_k =O(\frac{1}{k^2})$. 

Similarly, if $\omega_\infty$ is a K\"ahler--Einstein metric with ${\rm Ric} (\omega_{\phi_\infty})=-\omega_{\phi_\infty}$, using the second-order asymptotic behavior of  Bergman kernels for the weight $\phi_{\infty}$ (cf.~Theorem~\ref{thm: Bergman}) we get
\begin{eqnarray*}
\K_{k\phi_{\infty}, \mu_{\phi_\infty}}e^{-k \phi_{\infty}}  &=&\K_{k\phi_{\infty},\omega_{\phi_{\infty}}^{[n]}}e^{-k \phi_{\infty}} \\
&=&\left(\frac{k  }{2\pi}\right)^n \left(1-\frac{n}{2k} + O\left(\frac{1}{k^2}\right)\right).
\end{eqnarray*}  
Therefore, for $d_k= \left(\frac{k  }{2\pi}\right)^n\max\left\{\frac{1}{2}, ( 1-\frac{n}{2k})\right\}$ and $k>n$, we have
\begin{equation}\label{est_kernel_KE}
\beta_k(\phi_{\infty})= \frac{1}{k}\log(\K_{k\phi_{\infty
}}/d_k)=\phi_{\infty}+ \varepsilon_k,
\end{equation} 
where $\varepsilon_k= O(\frac{1}{k^3}).$

\medskip
Next, it follows from the fact that $\phi-\phi_{\infty}\geq C_1$ for some constant $C_1$, the extremal characterization of Bergman kernel implies that 
\begin{equation}\label{eq: est_ker}
\K_{k\phi }\geq e^{(k-1)C_1}\K_{k\phi_{\infty}}.
\end{equation}
Indeed, if we pick a section $0\neq  \sigma\in H^0{(X,L^k)}$ with $\int_X|\sigma|^2e^{-k\phi_\infty}\mu_{\phi_\infty}<+\infty$ then  for any $x\in X$,
\begin{equation*}
    \frac{|\sigma(x)|^2}{\int_X |\sigma|^2e^{-k\phi}\,{\rm d}\mu_{\phi}}\geq   e^{(k-1)C_1}\frac{|\sigma(x)|^2}{\int_X |\sigma|^2e^{-k\phi_\infty}\,{\rm d}\mu_{\phi_\infty}}.
\end{equation*}
Taking the supremum, we obtain the inequality~\eqref{eq: est_ker}. 
Therefore, combining with \eqref{est_kernel}  (reps. \eqref{est_kernel_KE}) we have
$$ \beta_k(\phi)-\phi_{\infty}\geq \frac{k-1}{k}C_1 - \varepsilon_k  $$
for some positive $\varepsilon_k=O(\frac{1}{k^2})$ (reps. $\varepsilon_k=O(\frac{1}{k^3})$).
Similarly, we have
$$ \beta_k(\phi)- \phi_{\infty}\leq \frac{k-1}{k}C_2 + \varepsilon_k  $$
as required.
\end{proof}

\begin{proof}[Proof of Theorems~\ref{thm: main} \& \ref{thm: construction-KE}]
We define by recurrence  $\phi_{k+1}:=\beta_k(\phi_k)$. 
For any $k\geq 1$, denote by $C_k$ the optimal constant such that
\begin{equation}
\phi_k-\phi_{\infty}\geq C_k.
\end{equation}
Lemma \ref{lem_key} thus implies that 
$$C_{k+1}\geq \frac{k-1}{k}  C_k-\varepsilon_k .$$
Since $\varepsilon_k$ is of order $1/k^2$, it follows that $k C_{k+1}$ is uniformly bounded from below by a quantity of order $O(\log k)$.
  Similarly, Lemma \ref{lem_key} yields a uniform upper bound for $k \widetilde C_{k+1}$, where $\widetilde C_k$ is the optimal constant such that 
$$\phi_k-\phi_{\infty}\leq  \widetilde C_k.$$
Since $C_k\leq \widetilde C_k$ by definition, we have $C_k=O\left(\frac{\log k}{k}\right)$ and $\widetilde C_k=O\left(\frac{\log k}{k}\right)$. 
Hence we obtain the desired convergence rate $\frac{\log k}{k}$.

By the same argument, we obtain the uniform convergence with rate $\frac{1}{k}$ when $\omega_{\phi_\infty}$ is a K\"ahler--Einstein metric.
\end{proof}



\section{Variation of K\"ahler--Einstein metrics}
\label{sect: pos-variation}
\subsection{Positivity of relative Bergman kernels}

{
\begin{setup}\label{setup: family}
Let $p:X \to Y$  be a holomorphic surjective map from a K\"ahler manifold $X$ of dimension $n+m$ to a complex manifold $Y$
of dimension $m$. Let $Y^\circ\subset Y$ denote the set of points that are not critical values of $p$ in $Y$ and set $X^\circ=p^{-1}(Y^\circ)$.  
\end{setup}

If $p:X\to Y$ is in Setup~\ref{setup: family}, then  the relative canonical bundle associated with $p $ is defined by $K_{X/Y}:= K_X-p^*(K_Y)$.  
It follows from \cite{Berndtsson09-curvature-fibrations, Berndtsson_Paun_08} that we have an isomorphism between $K_{X_y}$ and $(K_X-p^*K_{Y})^{}|_{X_y}= K_{X/Y}|_{X_y}$ {for $y\in Y^\circ$}.
Let $t=(t_1, \ldots,t_m)$ be a local coordinates near a point in $Y$. Then these local coordinates define a local $(m, 0)$-form $dt= dt_1\wedge\ldots \wedge dt_m$, which trivializes $K_Y$. This trivialization induces a local section $\tilde u$ of $K_X$ over $X_y$ by $\tilde u = u \wedge p^*(dt) = u \wedge d\tilde t$,where $\tilde t_j = t_j \circ p$, $j=1,\ldots,m$, are functions on $X$.
Conversely, any local section $\tilde u$ of $K_X$ can be written as $\tilde u = u \wedge d\tilde t$ locally. Then, the restriction of $u$ to the fibers $ X_y$ is uniquely defined as a local section of $K_{X_y}$. Thus the restriction of $u$ to the fiber $X_y$ is well defined as a local section of $K_{X_y}$, which we call the trace of $\tilde u$ and denote by $\tau_y(\tilde u)$.
Then the isomorphism is given by
\begin{equation}
\begin{split}
    K_{X_y}+ p^*(K_{Y})|_{X_y}& \rightarrow K_{X}|_{X_y} \\
    u\otimes p^* {\rm d}t &\mapsto u \wedge {\rm d}\tilde t  
\end{split}    
\end{equation}
with inverse map $\tilde u\mapsto \tau_y(u)\otimes p^*( dt)$. Hence the induced isomorphism  $K_{X_y}\to K_{X/Y}|_{X_y}$ is given by $u\mapsto u\wedge p^*(dt)/dt$.

Let $\omega$ be a relative K\"ahler metric on $X$ such that $\omega_y:=\omega|_{X_y}$ is a K\"ahler metric on $X_y$ for $y\in Y^\circ$. 
 Assume that $L$ is a line bundle on $X$ and $h_L=e^{-\phi_L}$ is a singular Hermitian metric on $L$ with positive curvature current. 
{For each $y\in Y^\circ$, we define  $E_y:=H^0(X_y,(K_{X_y}+ L|_{X_y})\otimes \mathcal{I}(\phi_L|_{X_y}))$}. Here, $\mathcal{I}(\phi)$ denotes the {\em multiplier ideal sheaf}  of holomorphic functions on X
that are square integrable against $e^{-\phi}$.
 Then the (singular) Hermitian metric $\phi_L$ on $L$ induces a Hilbert norm $\|\cdot\|_{y}$ on $E_y$ by
\begin{equation}
\|u\|^2_y:=\int_{X_y}c_nu\wedge \bar u e^{-\phi_L}=\int_{X_y}|u|_\omega^2e^{-\phi_L}\,{\rm d}V_{X_y,\omega}. 
\end{equation}
Let $x\in X^\circ$, and $y=p(x)$. Let $z_1, \ldots, z_{n+m}$ be local coordinates for a neighborhood $V\subset X^\circ$ of $x$ and $t_1, \ldots, t_m$ local coordinates near $y$. For $u\in H^0(X_y, K_{X_y}+L|_{X_y})$, we define $u'$  by $$u'{\rm d}z=u\wedge p^* {\rm d}t.$$
  The fiberwise Bergman kernel $\K(x)$ of $(K_{X/Y}+ L)|_{X^\circ}$ with respect to this norm is defined by 
\begin{equation}\label{eq_bergman_twist}
      \K(x)= \sup_{u\in E_y, \|u\|^2_y\leq 1  } |u'(x)|^2.
  \end{equation}
  We remark that this concept of the Bergman kernel coincides with the classical one.
Thus, $\phi_{X/Y}(x):= \log \K(x)$ provides a Hermitian metric $h=e^{-\phi_{X/Y}}$ on $K_{X/Y}$. 

We establish a pseudo‑effectivity criterion for the relative adjoint line bundle $K_{X/Y}+L$ by proving the log‑plurisubharmonicity of the fiberwise Bergman kernel. The only assumptions needed to obtain a nontrivial positively curved metric are the positivity of $L$ together with the existence of an $L^2$ section of $K_{X/Y}+L$. The following theorem is a non‑compact extension of \cite[Theorem~0.1]{Berndtsson_Paun_08} (see also \cite[Theorem 4.1.1]{Paun_Takayama_2018}, \cite[Theorem 1.3]{Zhou_Zhu_L_2_jdg}, \cite[Theorem 3.5]{cao17-OTtheorem} and \cite[Theorem 1.4]{Deng-Wang-Zhang-Zhou24-psh}).
\begin{theorem}\label{thm_variation}
   Let $X$ be a weakly pseudoconvex K\"ahler manifold of dimension $n+m$ and $Y$ a complex manifold of dimension $m$.
Let $p: X\to Y$ be a surjective holomorphic map.
    Let $L$ be a line bundle on $X$ equipped with a singular Hermitian metric $h_L = e^{-\phi_L}$ satisfying:
    \begin{itemize}
        \item  $dd^c \phi_L\geq 0$ in the sense of currents;
        \item  $H^0(X_{y_0}, (K_{X_{y_0}} + L|_{X_{y_0}})\otimes \mathcal{I}(\phi_L|_{X_{y_0}})) \neq 0$ for some $y_0 \in Y^\circ$. 
    \end{itemize} 
    Denote by $\K_{\phi_L}$ the relative Bergman kernel of $(K_{X/Y}+L)|_{X^\circ}$. Then $\K_{\phi_L}$ has positive curvature current and extends across $X\setminus X^\circ$ to a metric with positive curvature current on all of $X$.
\end{theorem}

A consequence of Theorem \ref{thm_variation} is the positivity of the relative Bergman kernel.
\begin{corollary} Let $X$ be weakly pseudoconvex K\"ahler manifold and $Y$ a complex manifold. 
Let $p: X\to Y$ be a surjective holomorphic map.
 Assume that $H^0(X_{y_0}, K_{X_{y_0}})\neq 0$ for some $y_0\in Y^\circ$. Then the relative pluricanonical line bundle $K^k_{X/Y}$ admits a metric with positive curvature current. Moreover, this metric coincides with the fiberwise $k$‑Bergman kernel metric on the fibers of $p$.
\end{corollary}
\begin{proof}
   First, take $L=\mathcal{O}_X$ with the constant metric. Then Theorem \ref{thm_variation} shows that $\phi^{(1)}_{X/Y}:=\log \K$ defines a Hermitian metric on $K_{X/Y}$ with positive curvature current. Applying Theorem \ref{thm_variation} inductively with $L=K^k_{X/Y}$ endowed with the Hermitian metric $\phi_{X/Y}^{(k)}$ for $k=1,2,\ldots$, we obtain a new metric $\phi_{X/Y}^{(k+1)}$ on $K_{X/Y}^{k+1}$ with positive curvature current, given by the logarithm of the corresponding Bergman kernel. 
\end{proof}


The proof of Theorem~\ref{thm_variation} follows the same strategy as in the compact case, cf. \cite{Berndtsson_Paun_08, Paun_Takayama_2018,Zhou_Zhu_L_2_jdg}. 
We provide a brief sketch for the reader’s convenience.
We will use the following version of the Ohsawa--Takegoshi extension theorem, which can be deduced from \cite[Theorem~1.1]{Zhou_Zhu_L_2_jdg} (cf. also \cite{cao17-OTtheorem, BL16-OTproof, GuanZhou-l2extension-pb,Xu-Zhou24-optimal-L2}).
\begin{lemma}\label{lem: OT-theorem} Let $(X,\omega)$ be a weakly pseudoconvex K\"ahler manifold. 
    Let $p:X\rightarrow \Delta_r=\{t: |t|<r\} \subset \mathbb{C}$ be a surjective holomorphic map. Suppose that $X_0=p^{-1}(0)$ is a submanifold of codimension $1$. Let $L$ be a holomorphic line bundle on $X$ equipped with a (possibly singular) Hermitian metric $\phi_L$ satisfying $\ddc\phi_L \ge 0$.
   Let $u \in H^0(X_0, K_{X_0} + L|_{X_0})$ with $\int_{X_0} |u|^2_\omega e^{-\phi_L}\, dV_{X_0} < \infty$. Then there exists $U_r \in H^0(p^{-1}(\Delta_r), K_{X/Y}+L)$ such that $U_r|_{X_0} = u \wedge p^*dt$ and
    $$\frac{1}{\pi r^2}\int_{p^{-1}({\Delta_r)}} |U_r|^2_\omega e^{-\phi_L}{\rm d}V_{X,\omega}\leq \int_{X_0} |u|^2_\omega e^{-\phi_L}\,{\rm d}V_{X_0,\omega}.$$
\end{lemma}

}

 \begin{proof}[Proof of Theorem \ref{thm_variation}]
 We follow the same argument as in the compact case (see \cite{Berndtsson_Paun_08, Paun_Takayama_2018,Zhou_Zhu_L_2_jdg}).
 Recall that for any $y\in Y^\circ$, setting $\phi_y=\phi_L|_{X_y}$,
 we define the $L^2$ norm on $u_y\in H^0(X_y,K_{X_y}+ L|_{X_y})$ by 
 \[\|u_y\|^2:=\int_{X_y} |u_y|^2_\omega e^{-\phi_y} \,{\rm d}V_{X_y,\omega}. \]
Let $(t,z)=(t_1,\ldots,t_m,z_1,\ldots,z_n)$ be local coordinates near a point $x\in X$. Set $y=p(x)$. 
 The Bergman kernel of $K_{X_y}+L|_{X_y}$ on $X_y$ ($y=y(t)$) is defined by
     \begin{equation}\label{eq: Bergman-Kernel}
       \K_{t}(z)=\sup \left\{|u_y'(z)|^2: u_y\in H^0(X_y,K_{X_y}+ L|_{X_y}),\|u_y\| \leq 1 \right\} 
     \end{equation}
     where $u_y=u'_y(z){\rm d}z\otimes e$ with $e$ a local frame of $L$ near $x$.

     Remark that  $p:X^\circ \to Y^\circ$ is a holomorphic submersion. 
     To show that $\log\K_t(z)$ is psh with respect to $(t,z)$, one needs to prove the subharmonicity of $\log\K_t(z)$ along any complex line. Without loss of generality, we can assume that $Y^\circ=\Delta_R\subset\mathbb{C}$ is a disk ($m=1$).  

\smallskip
The first property is the upper semicontinuity of $\log \K_t(z)$ with respect to $(t,z)$.
Indeed, thanks to Demailly's global regularization \cite[Section 9]{demailly82-estimations}, we can find an increasing sequence of smooth metrics $(h_\varepsilon = e^{-\phi_\varepsilon})_{\varepsilon>0}$ on $L$, converging to $h_L = e^{-\phi_L}$ in $L^1$ and satisfying
$\ddc\phi_\varepsilon\geq -\lambda_\varepsilon \omega$, $\ddc \phi_\varepsilon \ge -\lambda_\varepsilon \omega$, where $(\lambda_\varepsilon)$ is an increasing family of continuous functions on $X$ with $\lambda_\varepsilon \searrow 0$ almost everywhere. The corresponding Bergman kernels $\K_{t,\varepsilon}$ are continuous. Arguing as in \cite[Theorem~1.1]{Paun_Takayama_2018}, we obtain that $\log \K_{t,\varepsilon}(z)$ decreases to $\log \K_t(z)$ for $(t,z)\in U$ as $\varepsilon \searrow 0$. This proves the desired upper semicontinuity.

\medskip 
We next show that $\log \K$ satisfies the sub-mean inequality on every holomorphic disc in $Y^\circ$, namely, for any $t_0\in Y^\circ$,
\[ \log \K_{t_0}(z)\leq \frac{1}{\pi r^2} \int_{\Delta_r}\log \K_{t}(z) \,i {\rm d}t\wedge {\rm d}\Bar{t}\] for some $r<<1$ and $\Delta_r=\Delta(t_0,r)\subset Y^\circ$. Fix $u_0\in H^0(X_{t_0},K_{X_{t_0}}+ L|_{X_{t_0}})$ such that $\|u_0\|=1$ and $\K_0(z)=|u_0'(z)|^2$. By the Ohsawa--Takegoshi type theorem (Lemma~\ref{lem: OT-theorem}), $u_0$ can be extended to a holomorphic section $U$ on the preimage of $\Delta_r$ such that 
\[ \frac{1}{\pi r^2}\int_{\Delta_r}\int_{X_t} |U'(t,z)|^2e^{-\phi_L} \,{\rm d}V_{X_t,\omega}\, i{\rm d}t\wedge {\rm d}\Bar{t}\leq \|u_0\|^2= 1, \]
where we have used the Fubini theorem and $U=U'dz\wedge p^*({\rm d}t)\otimes e$.
Since $U(t,z)$ is holomorphic in $t$, we can write 
\[ \log \K_{t_0}(z)=\log |U'(0,z)|^2\leq \frac{1}{\pi r^2}\int_{\Delta_r}\log |U'(t,z)|^2 i {\rm d}t\wedge {\rm d}\Bar{t}.\]
Using the extremal property of the Bergman kernel, the term on the right-hand side is dominated by
\[ \frac{1}{\pi r^2}\int_{\Delta_r}\left(\log \K_t(z)+\log \int_{X_t} |U'(t,z)|^2e^{-\phi_L} {\rm d}V_{X_t,\omega} \right)\, i {\rm d}t\wedge {\rm d}\Bar{t}.\]
By Jensen's inequality, we obtain 
\begin{align*}
    \frac{1}{\pi r^2}\int_{\Delta_r} \log \int_{X_t}& |U'(t,z)|^2e^{-\phi_L} {\rm d}V_{X_t,\omega}\, i{\rm d}t\wedge {\rm d}\Bar{t}\\
    &\leq\log \left( \frac{1}{\pi r^2}\int_{\Delta_r}\int_{X_t} |U'(t,z)|^2e^{-\phi_L} {\rm d}V_{X_t,\omega}\, i{\rm d}t\wedge d\Bar{t}\right) \leq 0.
\end{align*}
This tells us that $\log\K_t(z)$ is psh in the horizontal direction. On the other hand, by the convexity of  $|u_t'(z)|$ and the definition of $\log \K$, we deduce that $\log\K(z)$ is also psh in the fiberwise direction. Therefore, $\log\K$ is psh on $p^{-1}(Y^\circ)$, and in particular, the curvature of the induced singular Hermitian metric is positive on $p^{-1}(Y^\circ)$ in the sense of currents.   

It remains to show that the curvature current of the induced Hermitian metric above can be extended over the whole of $X$. It suffices to show that its local potential is bounded from above near the singular fibers. 

Fixing an open subset $V\subset X$, we define
the function $\K$ on $V\cap X^\circ $ as~\eqref{eq: Bergman-Kernel}.   
The problem is equivalent to proving that $\K$ is uniformly bounded from above on $V\cap X^\circ$. For any $x\in V\cap X^\circ$, there exists a section $\sigma\in H^0(X_{p(x)}, K_{X_{p(x)}}+L|_{X_{p(x)}})$ such that
\[ \K(x)=|\sigma'(z)|^2,\quad\int_{X_{p(x)}}|\sigma|^2 e^{-\phi_{p(x)}}\, {\rm d}V_{X_{p(x)},\omega}=1,\] 
where $(t,z)$ local coordinates near $x$ and $\sigma=\sigma' dz\otimes e$.
Using the Ohsawa--Takegoshi extension theorem, we can find an extension $\Tilde{\sigma}$ of $\sigma$ on $p^{-1}(p(V))$ such that
\[ \int_{p^{-1}(p(V))}|\Tilde{\sigma}|^2 e^{-\phi_L}\, {\rm d}V_{X,\omega}\leq C_V \]
with $C_V>0$ only depending on $V$. It follows from the sub-mean value inequality that $|\sigma'    (z)|^2$ is dominated by a constant only depending on $C_V$. 
This completes the proof.
 \end{proof}
 

\subsection{Plurisubharmonic variation of K\"ahler--Einstein metrics} 
This section is devoted to studying the plurisubharmonic variation of K\"ahler--Einstein metrics on non‑compact complex manifolds. The idea goes back to Tsuji~\cite{Tsuji10-construction-KE} and Schumacher~\cite{Schumacher12-positivity-bundles}, who studied holomorphic families of canonically polarized, compact Kähler manifolds using two different methods. Schumacher’s method is based on a maximum principle, which moreover yields strict plurisubharmonicity when the family is not infinitesimally trivial. Tsuji’s approach shows that K\"ahler--Einstein metrics can be obtained via an iteration scheme involving suitably normalized pluricanonical sections, as explained in this paper.In \cite{Tsuji13-construction-KE-pseudoconvex}, Tsuji extended his method to families of K\"ahler--Einstein metrics on bounded pseudoconvex domains, using Berndtsson’s results~\cite{Berndtsson09-curvature-fibrations, Berndtsson06-subharmonicity-bergman}. For alternative proofs of this result on strongly pseudoconvex domains based on Schumacher’s method, see \cite{Choi15-variations-pseudoconvex, Choi15-study}. In this work, we follow Tsuji’s method to obtain the following.

\begin{theorem}\label{thm: variation-psh} Let $p:X\to Y$ be in Setup~\ref{setup: family}.  Assume further that $X$ is weakly pseudoconvex, and 
\begin{itemize}
    \item[(i)] for every $y\in Y^\circ$, $X_y$ admits a K\"ahler--Einstein metric $\omega_{\mathrm{KE},y}$ with negative Ricci curvature such that $(X_y,\omega_{\mathrm{KE},y})$ has bounded geometry of order $\ell\ge 5$, locally uniformly in $y$, i.e., for any neighborhood $U\Subset Y^\circ$ the constants $r,c,A_k$ in Definition~\ref{def:bd_geo} depend only on $U$. 
    
    \item[(ii)]  $K_{X/Y}$ admits a singular Hermitian metric $h_{0}=e^{-\phi_{0}}$ with positive curvature current and for every $y_0\in Y^\circ$, there exists a neighborhood $U\Subset Y^\circ$ of $y_0$ such that $\|\phi_{0,y}-\phi_{\mathrm{KE},y}\|_{L^\infty(X_y)}\le C(U)$ for all $y\in U$. 
\end{itemize}
Then the Hermitian metric on $(K_{X/Y})|_{X^\circ}$ induced by the fiberwise K\"ahler--Einstein metrics is positively curved.  Moreover, if we have the additional condition that
\begin{itemize}
\item[(iii)] the K\"ahler--Einstein potentials $(\phi_{\mathrm{KE},y})_{y\in Y^\circ}$ are bounded from above near the singular fibers,
\end{itemize}
then this metric extends canonically across $X\setminus X^\circ$ to a positively curved metric.
\end{theorem}
The condition (ii) holds whenever one can compare the Bergman kernel of $H^0(X_y,K_{X_y})$ with the K\"ahler--Einstein measure $\omega_{\mathrm{KE},y}^n$. Indeed, in this case, the fiberwise Bergman kernel for $H^0(X_y, K_{X_y})$ gives us a Hermitian metric $\phi_0$ on $K_{X/Y}$ which satisfies $(ii)$.  The following corollary illustrates this using the comparison between the Bergman kernel and the K\"ahler--Einstein measure in Greene--Wu \cite{Greene_Wu}. 

\begin{corollary}
 Let $p:X\to Y$ be in Setup~\ref{setup: family}.  Assume further that  $X$ is weakly pseudoconvex and that
\begin{itemize}
    \item[(i)] for every $y_0\in Y^\circ$, there exist a neighborhood of $y_0$, $U\Subset Y^\circ$ and $0<A_U< B_U$  such that for any $y\in U$, $X_y$ admits a K\"ahler metric $\omega_y$ whose holomorphic sectional curvature belongs to $(-B_U, -A_U)$;   
    \item[(ii)] $X_y$ is simply-connected. 
\end{itemize}
Then, the curvature current of the Hermitian metric on $(K_{X/Y})|_{X^\circ}$ induced by the K\"ahler--Einstein metrics on the fibers is positive. 
\end{corollary}
\begin{proof}
  It follows from \cite[Corollary 7]{Wu-Yau20invariant} thatfor any $y\in U$, the fiber $X_y$ admits a unique K\"ahler--Einstein metric with bounded geometry, uniformly equivalent to $\omega_y$, with equivalence constants depending only on $A_U$, $B_U$, and $\dim X_y$. Moreover, by \cite[p. 144, Theorem H (3)]{Greene_Wu} the fiberwise Bergman kernel $\K_y$ satisfies $A_1\,\omega_y^n \le \K_y \le A_2\,\omega_y^n$ for some positive constants $A_1,A_2$ depending only on $A_U$, $B_U$, and $\dim X_y$.  thus the fiberwise Bergman kernel defines a Hermitian metric $\phi_0$ on $K_{X/Y}$ satisfying (ii) in Theorem~\ref{thm: variation-psh}. 
\end{proof}
 \begin{proof}[Proof of Theorem~\ref{thm: variation-psh}]  
 Fix $y\in Y^\circ$. We first note that $H^0(X_y, K_{X_y}^{k_0+1}\otimes \mathcal{I}(k_0\phi|_{X_y}))\neq 0$ for some $k_0>0$. By the H\"ormander $L^2$ estimate, there exists $k_0>0$ depending only on the bounded geometry of $(X_y,\omega_{\mathrm{KE},y})$ such that
$H^0(X_{y}, K^{(k_0+1)}_{X_{y}} \otimes \mathcal{I}(k_0\phi_{ \rm KE}|_{X_{y}}))\neq 0$. By $(ii)$ we have $\| \phi_{0,y}-\phi_{{\rm KE}, y}\|\leq C$,  hence $H^0(X_{y}, K^{(k_0+1)}_{X_{y}} \otimes \mathcal{I}(k_0\phi_{}|_{X_{y}}))\neq 0$. 

We now adapt the argument  in~\cite{Tsuji10-construction-KE}. 
Set $L_{k_0}:=K_{X/Y}^{k_0}$ and $h_{k_0}=e^{-\phi_{L_{k_0}}}$ with $\phi_{L_{k_0}}:=-k_0\phi_0$, which has positive curvature current.  
 By iteration, at step $k\geq k_0$, assume that  $L_k:=K_{X/Y}^k|_{X^\circ}$ admits a Hermitian metric $h_k=e^{-k\phi_k}$ with $ \ddc\phi_k\geq 0$. Define $\phi_{k+1}:=\frac{1}{k}\log(\K_{k\phi_k}/d_k)$, where $\K_{k\phi_k}$ is the Bergman kernel of $(K_{X/Y}+L_k)|_{X^\circ}$. On each fiber, set $\phi_{k+1,y}=\beta_k(\phi_{k,y})=\frac{1}{k}(\log \K_{k\phi_{k,y}}/d_k)$. By Theorem~\ref{thm: Bergman}, $\|\phi_{k,y}-\phi_{\mathrm{KE},y}\|\le C$, hence $H^0(X_y,(K_{X_y}+L_k|_{X_y})\otimes \mathcal{I}(\phi_{L_k}|_{X_y}))\neq 0$. Applying Theorem~\ref{thm_variation}, we conclude that $\ddc\phi_{k+1}\ge 0$.

 \smallskip
     For any $y_0\in Y^\circ$, choose $U\Subset Y^\circ$ containing $y_0$. By the hypothesis and Theorem~\ref{thm: construction-KE}, $\|\phi_{k, y'} -\phi_{{\rm KE},y'}\|_{L^\infty(X_{y'})}\leq C(U)/k$ for all $y'\in U$,  hence $\|\phi_{k} -\phi_{\rm KE}\|_{L^\infty(p^{-1}(U))}\leq C(U)/k$. By locally uniform convergence, the limit $\phi_{\mathrm{KE}}:=\lim_k \phi_k$ defines a singular Hermitian metric on $(K_{X/Y})|_{X^\circ}$ with positive curvature current $T:=\ddc\phi_{\mathrm{KE}}$. 
\end{proof}
\subsection{Families of K\"ahler--Einstein metrics on bounded pseudoconvex domains}

We recall some properties of pseudoconvex domains (cf. \cite{ChengYau80-complete}).  Let $\Omega$ be a strongly pseudoconvex domain in $\mathbb{C}^n$.  Let $\rho$ be a smooth defining function for $\Omega$, i.e., $\rho\in\mathcal{C}^\infty(\overline{\Omega})$, $\Omega=\{\rho<0\}$, $\partial\rho\neq 0$ on $\partial\Omega$, and $(\rho_{i\Bar{j}})>0$ in $\overline\Omega$.  We define $\varphi_\Omega = -(n+1)\log (-\rho)$, which is a strictly plurisubharmonic function on $\Omega$ and $\varphi_\Omega(z) \rightarrow +\infty$ as $z\rightarrow \partial \Omega$.  It follows from \cite{ChengYau80-complete} that  $(\Omega, \omega_{\varphi_\Omega
}) $ is a complete K\"ahler manifold, where $\omega_{\varphi_\Omega}= dd^c\varphi_\Omega$. Since 
$$\det ((\varphi_\Omega)_{i\bar j }) = (n+1)^n \left(-\frac{1}{\rho}\right)^{n+1} \det(\rho_{i\bar j })(-\rho +|d\rho|^2),$$
 we have the Ricci curvature
$${\rm Ric}(\omega_{\varphi_\Omega}) = - (\varphi_\Omega)_{i\bar j }+ \partial_i\partial_{\bar j}F ,$$
where  $F=-\log [\det(\rho_{\bar j i})(-\rho +|d\rho|^2)]$ is smooth  in $\Omega$ and bounded on $\overline \Omega$.  

\medskip
As explained in \cite{ChengYau80-complete} that $(\Omega, \omega_{\varphi_\Omega})$  has bounded geometry. The following theorem is due to Cheng--Yau \cite{ChengYau80-complete,  Mok-Yau83-completeness}.
\begin{theorem}
\label{thm_CY}
There exists  a unique complete K\"ahler--Einstein metric $\omega_{\rm KE}=dd^c\phi_{\rm KE}$ with $-(n+1)$ scalar curvature
which solves
\begin{eqnarray}
\label{KE_eq}
(\ddc \phi_{\rm KE})^n= e^{(n+1)(\phi_{\rm KE} -\phi_\Omega)
+F}\omega_{\varphi_\Omega}^n,
\end{eqnarray}
and $|\nabla^k_{\varphi_\Omega}( \phi_{\rm KE}-\varphi_\Omega) |\leq C_k$ for any $k\in \mathbb{N}$.
Moreover, $(\Omega, \omega_{\rm KE})$ has bounded geometry of any order.
\end{theorem}

Let $p:D\subset\mathbb{C}^{n}\times\mathbb{C}^m\to B\subset \mathbb{C}^m$ be the second projection where $D$ is a strictly pseudoconvex domain (with smooth boundary), $B$ is an open ball. Suppose that for each $t\in B$, $D_t:=p^{-1}(t)$  is bounded strongly pseudoconvex domain with
smooth boundary. It follows from~\cite{ChengYau80-complete,  Mok-Yau83-completeness} (cf. Theorem \ref{thm_CY}) that on each
 slice $D_t$ there exists a unique complete K\"ahler form $\omega_{\phi_t}$ such that 
\[ \Ric(\omega_{\phi_t})=-(n+1)\omega_{\phi_t}.\]
Then we  have the  following corollary of Theorem \ref{thm: variation-psh}:
\begin{theorem}\label{thm_var_psh_domain}
    The metric $\phi_t$ induces a Hermitian metric on the relative line bundle $K_{D/B}$ with positive curvature current 
\end{theorem}
The result was proved by  Choi \cite{Choi15-variations-pseudoconvex,Choi15-study} using the boundary behavior of the geodesic curvature, which satisfies a certain elliptic equation, and by  Tsuji \cite{Tsuji13-construction-KE-pseudoconvex} using his construction of K\"ahler--Einstein metrics. 

\begin{proof}
Let $(z,t)\in\mathbb{C}^n\times\mathbb{C}^m$ be the standard coordinates. We denote by $\rho$ the defining function of $D$ which satisfies \begin{itemize}
    \item $\rho\in\mathcal{C}^\infty(\overline D)$ and $D=\{(z,t): \rho(z,t)<0\}$;
    \item $\partial\rho\neq 0$ on $\partial D$,
    \item  $(\rho_{i\Bar j}(\cdot,t))>0$ in $\overline  D$.
\end{itemize} This implies that $\varphi_D(\cdot, t)=-\log (-\rho(\cdot, t))$ is a strictly smooth plurisubharmonic function in $D_t$ for each $t\in B$. 
Applying Theorem~\ref{thm: variation-psh}, it suffices to show that     $\|\phi_{t}-\varphi_{D_t}\|\leq C$ independent of $t$. Setting $u_t=\phi_t-\varphi_D(\cdot,t)$, it solves
\[ \det ((\varphi_D)_{i\Bar j}(\cdot, t)+(u_t)_{i\Bar j})= e^{(n+1)u+F}\det ((\varphi_D)_{i\Bar j}(\cdot, t))\] in $D_t$ where the function $F$ is bounded on $\overline D$ (defined as above). Applying Theorem~\ref{thm: variation-psh}, it suffices to show that $\|u_t\|_{L^\infty(D_t)}\leq C$ uniformly in $t$.
Indeed, it follows from the maximum principle~\cite[Proposition 4.1]{ChengYau80-complete} that $$\sup_{D_t}|u_t|\leq \frac{1}{n+1}\sup_{D_t}|F|\leq C$$ for a constant $C>0$ uniformly in $t$. This finishes the proof. 
\end{proof}

\subsection{Families of K\"ahler--Einstein metrics with  cusp singularities}
 Let $p:\overline{X} \to Y$ be a holomorphic subjective proper map between compact K\"ahler manifolds with relative dimension $n$, $D=\sum_{i=1}^\ell D_i$  a reduced divisor with  simple normal crossing support. Let $X:=\overline{X}\setminus D$. Assume $Y^\circ$ is the largest Zariki open subset of $Y$ such that if $\overline{X}^\circ:=p^{-1}(Y^\circ)$ then every fiber ${\overline X}_y:= p^{-1}(\{y\})$  is smooth, $D_y:=D\cap {\overline{X}_y}$ has simple normal crossing support. We denote $X^\circ=\overline{X}^\circ \setminus D $.


Assume that $K_{\overline{X}_y}+D_y$ is ample for any $y\in Y^\circ$. By the well-known result of Kobayashi \cite{Kobayashi84-KE-open} and Tian--Yau~\cite{TianYau87-KE-complete}, building upon Cheng--Yau’s strategy~\cite{ChengYau80-complete}, for each fiber, there exists a (unique) complete K\"ahler metric $\omega_y$ on $X_y:= \overline{X}_y\backslash D_{y}$ with cusp singularities along $D_{y}$ and $\Ric(\omega_y)=-\omega_y$ on $X_y$.  The metric $\omega_y$ can be extended as a positive current on $\overline{X}_y$, which satisfies \begin{equation}\label{eq: KE-cusp1}
    \Ric(\omega_y)=-\omega_y+[D_{y}]
\end{equation} in the sense of currents on $\overline{X}_y$, where $[D_y]$ is the integration current along $D_y$. 

Then we will see in  Theorem \ref{thm: variation-psh2} below that a  similar statement in  Theorem \ref{thm: variation-psh} also holds in this setting, hence we get the following application:
\begin{theorem}  Assume the setup above, then the fiberwise cusp K\"ahler--Einstein metrics $\omega_y$ satisfying~\eqref{eq: KE-cusp1} can be glued to define a current $T$ on  $\overline{X}$ such that: 
    \begin{enumerate}[label=(\roman*)]
        \item $T$ is positive on $X^\circ= \overline{X}^\circ\backslash D$ ;
        \item $T$ extends canonically across $(X\backslash X^\circ)\cup D$ to a closed positive current.
    \end{enumerate}
\end{theorem}

\begin{proof}
If one chooses a background K\"ahler metric $\omega$ on $\overline{X} $ then it induces a Hermitian metric $h^\omega_{\overline{X}/Y}=e^{-\phi_\omega}$ on $K_{\overline{X}/Y}$ with curvature form { $\Theta_{h^\omega}(K_{\overline{X}/Y})=dd^c \phi_\omega$}, cf.~\cite[\S 3.1]{paun17-relative-nef}. 
Fix smooth Hermitian metrics $h_i=h_{D_i}=e^{-\phi_i}$ on $\mathcal{O}_{\overline{X}}(D_i)$ with curvature form $\theta_{D_i}=dd^c\phi_i\in c_1(D_i)$. Denote $\theta_{i,y}=\sum\theta_{D_i|_{\overline{X}_y}}$ for every generic point $y\in Y^\circ$. Let $s_i$ be a holomorphic section cutting out $D_i$, and denote $|s_i|^2=|s_i|^2_{h_i}$ and $|s|^2=\prod|s_i|^2 $. 
On each fiber $\overline{X}_y$, we set $\eta_y=\Theta_{h^\omega}(K_{X_y})+\sum\theta_{D_i|_{X_y}}$ being chosen to be K\"ahler (the restriction of $\Theta_{h^\omega}(K_{\overline{X}/Y})+\sum\theta_{D_i}$).
The equation~\eqref{eq: KE-cusp1} can be written as
\begin{equation}
    \label{eq: cusp-MA} (\eta_y+\ddc\varphi_y)^n=|s|^{-2}{e^{\varphi_y}\omega^n},
\end{equation} where $\varphi_y$ is the unknown function. By~\cite[Theorem 4.2]{Berman-Guenancia14-logcanonical} (see also~\cite[Theorem 1.2]{dang2023kahler}), there exists a unique $\varphi_y\in \mathcal{E}^1(\overline {X}_y,\eta_y)$, which is the space of functions with finite energy. {We can glue the functions $\varphi_y$ to obtain a function $\varphi$ on $\overline{X}^\circ$. We will show that $\varphi$ defines a   Hermitian metric $\phi_{\rm KE}:= \phi_\omega+\sum_i\phi_i+ \varphi$ on  $(K_{X/Y})|_{X^\circ}$, with positive curvature current $T=dd^c\phi_{\rm KE},$ and $\phi_{\rm KE}$ extends across $X\backslash X^\circ$ to a singular Hermitian metric with closed positive current by applying Theorem~\ref{thm: variation-psh2} below, which is a variant of Theorem~\ref{thm: variation-psh}.}

\smallskip
It is convenient to work with the complete K\"ahler metric with Poincar\'e singularities $\omega_{P,y}=\eta_y-\sum_i\ddc\log(\log|s_i|^2)^2$ on $X_y=\overline{X}_y\backslash D_{y}$ (up to scaling the hermitian metrics on $\mathcal{O}_{\overline{X}}(D_i)$). It indeed defines 
a smooth K\"ahler metric with  Poincar\'e growth along $D_y$, having
bounded geometry at any order on $X_y=\overline{X}_y\backslash D_{y}$ as follows from~\cite[\S 1]{guenancia14-kahler-mixed}
(cf. also~\cite[Lemma 2]{Kobayashi84-KE-open}); all of these properties are satisfied uniformly in $y$. Setting $\psi_y=\varphi_y+\sum_i\log(\log|s_i|^2)^2$,
we have
\[ (\omega_{{\rm P},y}+\ddc \psi_y)^n= e^{\psi_y+F}\omega_{{\rm P},y}^n\]
where $F=\log\left( \frac{\omega^n} {\prod_i|s_i|^2\log|s_i|^2\omega_{\rm P}^n}\right)$ is known to be a bounded smooth function on $\overline{X}\backslash D$; cf.~\cite[Lemma 1.6]{guenancia14-kahler-mixed} or~\cite[Proposition 2.1]{carlson-Griffiths72}. 

For any $y_0\in Y^\circ$, take a neighborhood $U\Subset Y^\circ$ of $y_0$. 
{By the result of Kobayashi \cite{Kobayashi84-KE-open} and Tian--Yau~\cite{TianYau87-KE-complete}, we have that $\|\psi_y\|_{L^\infty( X_{y})}\leq C(U)$ for all $y\in U$.  Indeed,
since  the bisectional curvature of $(X_y,\omega_{{\rm P},y})$ is bounded below by some constant  depending on $U$ for all $y\in U$ (cf. \cite[Section 4]{guenancia14-kahler-mixed}), the Yau generalized maximum principle (see~\cite[Propotion 4.1]{ChengYau80-complete}) yields \[\sup_{X_y}|\psi_y|\leq \sup_{X_y} |F|\leq C(U) \]
for any $y\in U$ because $F$ is bounded on $\overline{X}\backslash D$.
}
Hence, for $\phi_0:= \phi_\omega+\sum_i \phi_i  -\sum_i\log(\log|s_i|^2)^2$
 and $\phi_{{\rm KE}}= \phi_\omega+\sum_i \phi_i+ \varphi$, we have
\[
\sup_{X_y}|\phi_{{\rm KE}, y}-\phi_{0,y}|= \sup_{X_y}|\psi_{y}|\leq C(U).
\]
Also,  the result of~\cite{Kobayashi84-KE-open,TianYau87-KE-complete} shows that
the K\"ahler--Einstein metric $dd^c \phi_{{\rm KE},y}=\eta_y+\ddc\varphi_y$ is complete and has bounded geometry on $X_y=\overline{X}_y\backslash D_{y}$ which is locally uniform in $y\in Y^\circ$. 

{With this setup, we apply Theorem~\ref{thm: variation-psh2}  to $\phi_0= \phi_\omega+\sum_i \phi_i  -\sum_i\log(\log|s_i|^2)^2$  and $(\phi_{{\rm KE}, y})_{y\in Y^\circ }$, to obtain a Hermitian metric $ \phi_{\rm KE}=\phi_\omega+\sum_i\phi_i+ \varphi$ on $(K_{X/Y})|_{X^0}$ which has a positive curvature current  $T=dd^c\phi_{\rm KE}$ on $X^\circ$. }
Moreover, we also have that $\phi_{{\rm KE}}$ is bounded from above near the singular fibers of $p|_X$.  This can be proved following the same argument as in \cite[Section 6]{Guenancia_2020}, which is based on the approach of P\u aun \cite[\S 3.3]{paun17-relative-nef}.  
 Therefore, we can extend $T$ to $X$ as a positive current. 
 
 {We conclude the proof by showing that $\phi = \phi_{\rm KE}$ extends across $D_{\rm reg}$, {where $D_{\rm reg}$ is the regular part of $D$}, and hence across $D$, following ~\cite[Lemma 3.7]{Berman-Guenancia14-logcanonical}. Fix $x \in D_{\rm reg}$, and let $(U, (z^1, \ldots, z^{n+m}))$ be local coordinates around $x$, where $\dim_{\mathbb{C}} \overline{X}=n+m$. In these coordinates, we write $D = (z^1 \cdots z^\ell = 0)$. It suffices to treat the case $\ell = 1$, as the general case can be proven similarly. We have the integral 
 \[\int_{U\backslash D}\frac{e^\phi \, {\rm d}V}{|z^1|^2} \leq C\int_{0<|z^1| < 1, |z^j|<1 }\frac{i{\rm d}z^1\wedge {\rm d}\bar z^1}{|z^1|^2\log^2 (|z^1|^2)}\wedge\bigwedge_{j=2}^{n+m} i{\rm d}
 z^j\wedge {\rm d}\Bar{z}^j< \infty.  \]  Suppose, by contradiction, that $\phi$ does not tend to $-\infty$ near $D$. Then, there exists a sequence $(z_k)$ in $U \setminus D$ converging to $D$ such that $\phi(z_k) \geq -C$ for some constant $C$. Let $\Delta_k$ denote the polydisk centered at $z_k$ with polyradius $(r_k, \delta, \ldots, \delta)$ for $r_k \searrow 0$ and some fixed $\delta > 0$, so that $\Delta_k \subset U \setminus D$. It follows from the sub-mean value inequality for $\phi$ at $z_k $ that 
\[ -C\leq \frac{1}{\Vol(\Delta_k)}\int_{\Delta_k}\phi\, {\rm d}V.\]
 Using Jensen's inequality, we obtain 
\[e^{-C(n)}\leq \int_{\Delta_k}\frac{e^\phi \,{\rm d}V}{r_k^2\delta^{2n}} \]
 and on $\Delta_k$, $|z^1|\leq 2r_k$ hence \[C'\leq\int_{\Delta_k}\frac{e^\phi \, {\rm d}V}{|z^1|^2}  \] with $C'=\frac{1}{4}e^{-C}\delta^{2n}$. Since the measure of $\Delta_k$ goes to zero as $k\to+\infty$ we infer that $\int_{U\backslash D}\frac{e^\phi dV}{|z^1|^2}$ is infinite, which is a contradiction.
 }
\end{proof}
 To end the section, we establish an extension result of the Ohsawa-Takegoshi type for the family of quasi-projective manifolds, which ensures that the statement in Theorem~\ref{thm_variation} and
    \ref{thm: variation-psh} holds in this case.
\begin{lemma}\label{OT-thm-quasi-proj}
Let $(\overline{X},\omega)$ be a K\"ahler manifold fibered over the disc $\Delta_r=\{t\in\mathbb{C}:|t|<r\}$ with $r>1$,  with compact fibers $\overline{X}_t$. 
Let $D$ be a reduced divisor with simple normal crossing support in $\overline{X}$ not contained in $\overline{X}_0$. Let $X=\overline{X}\backslash D$ and  $X_t=\overline{X}_t\backslash D$.
Let $L$ be a holomorphic line bundle on $\overline{X}$ equipped with a singular Hermitian metric $e^{-\phi}$ with semi-positive curvature. Let $u\in H^0(X_0, K_{X_0}+L|_{X_0})$ such that $\int_{X_0}u\wedge\bar ue^{-\phi}<\infty$. Then there exists $U\in H^0(X,K_X+L|_X)$ such that
$U=u\wedge dt$ for $t=0$ and \begin{equation}
    \label{eq: OT-inequ}
 \frac{1}{\pi r^2}\int_{X}U\wedge \Bar{U}e^{-\phi}\leq \int_{X_0}u\wedge \bar ue^{-\phi}.\end{equation}
\end{lemma}

\begin{proof}
    The proof is almost identical to that of~\cite[Theorem 1.3]{cao17-OTtheorem},  so we will be somewhat brief and focus on the differences.   
    We observe that $X=\overline{X}\backslash D$ admits a complete metric $\omega_{\rm P}$ with Poincar\'e singularity type; cf.~\cite{Kobayashi84-KE-open}.  
    We take a covering of $\overline{X}$ by finite open sets $V_\alpha$ biholomorphic to the polydisc in $\mathbb{C}^n$. For each $\alpha$ we fix a local coordinates $(t,z^1_\alpha,\ldots,z^n_\alpha)$ such that $\overline{X}_0\cap V_\alpha=(t=0)$ and $D\cap V_\alpha=(z_\alpha^1\cdots z_\alpha^\ell=0)$. So that $X\cap V_\alpha\simeq (\Delta^*)^\ell\times\Delta^{n+1-\ell}$ and 
    $ X_0\cap V_\alpha\simeq (\Delta^*)^\ell\times\Delta^{n-\ell}$ where $\Delta$ is the unit disc in $\mathbb{C}$ and $\Delta^*=\Delta\backslash \{0\}$.
    Locally on $U_\alpha$, we can write $u=f_\alpha\otimes e_\alpha$ where $f_\alpha$ is a holomorphic function on $X_0\cap V_\alpha$. We write the Hermitian metric $h=e^{-\phi}=h_0e^{-\varphi}$ for $h_0$ a fixed smooth Hemitian metric and a function $\varphi$ on $\overline{X}$. 
Thanks to Demailly's approximation theorem, there exists an increasing sequence $\{h_k\}=\{h_0e^{-\varphi_k}\}$ of metrics with analytic singularities with curvature $i\Theta_{h_k}(L)\geq -\frac{\omega}{k}$.

\smallskip
    \textbf{Step 1.} We claim that there exists a smooth section $\Tilde u$ on $K_X+L|_X$ such that $\Tilde{u}|_{X_0}=u$, $\Bar{\partial} \Tilde{u}|_{X_0}=0$ and for some $\sigma>0$, 
    \[ \int_X\frac{|\Bar{\partial}\Tilde{u}|^2_{\omega_{\rm P},h_0}e^{-(1+\sigma)\varphi}}{|t|^2(\log|t|)^2} \,{\rm d}V_{X,\omega_{\rm P}} \leq C\int_{X_0} |u|^2_{\omega_{\rm P},h}\, {\rm d}V_{X_0,\omega_{\rm P}}. \]
Thanks to  the strong openness~\cite{Guan-Zhou15-strong-openness}, we can find $\sigma>0$ such that
\[\int_{X_0\cap V_\alpha}|u|^2_{\omega_{\rm P},h_0}e^{-(1+\sigma)\varphi}\, {\rm d}V_{X_0,\omega_{\rm P}} \leq  2\int_{X_0\cap V_\alpha}|u|^2_{\omega_{\rm P},h} \, {\rm d}V_{X_0,\omega_{\rm P}} \]
    We apply the usual version of the Ohsawa–Takegoshi theorem~\cite[Theorem 13.6]{Deamilly-book-analytic} with the weight $e^{-(1+\sigma)\varphi}$, we can find holomorphic sections $\Tilde f_\alpha$ of $K_X+L|_X$ on $U_\alpha$ extending $u$ such that 
    \[\int_{V_\alpha\backslash D}\frac{|\Tilde f_\alpha|^2_{\omega_{\rm P},h_0}e^{-(1+\sigma)\varphi}}{|t|^2(\log|t|)^2}\, {\rm d}V_{X,\omega_{\rm P}} \leq\int_{X_0\cap V_\alpha} |u|^2_{\omega_{\rm P},h} \,{\rm d} V_{X_0,\omega_{\rm P}}. \]
Let $\{\chi_\alpha\}$ be a partition of unity
subordinate to $\{V_\alpha\}$. We can check $\Tilde{u}:=\sum_\alpha\chi_\alpha \Tilde{f}_\alpha$ satisfies the claim.

\smallskip
\textbf{Step 2.}   
Set $\Psi=\log|t|^2$ and $g_m=\Bar{\partial}((1-b_m'\circ\Psi)\Tilde u)$ where $b_m$ is a $\mathcal{C}^1$ function defined on $\mathbb{R}$ such that \[b_m(s)=s \;\text{for}\; s\geq -m,\quad b_m''(s)=\mathbf{1}_{\{-m-1\leq s\leq -m\}}(s).\]
We claim that there exists a sequence $\{a_m\}\subset\mathbb{N} $ tending to $\infty$, $L$-valued locally
$L^2$-forms $\gamma_m$ and $\beta_m$ such that
\[\Bar{\partial}\gamma_m+(\frac{m}{a_m})^{\frac{1}{2}}\beta_m=g_m,\quad\lim_{m\to\infty}\frac{m}{a_m}=0 \]
and for $k\geq a_m$
\begin{equation}\label{eq: OT2}
   \varlimsup _{m\to\infty} \int_{X}{|\gamma_m|^2_{\omega_{\rm P},h_{a_m}}}\, {\rm d}V_{X,\omega_{\rm P}}\leq \pi r^2\int_{X_0}|u|^2_{\omega_{\rm P},h_0}e^{-\varphi_{k}}\,{\rm d}V_{X_0,\omega_{\rm P}}
\end{equation}
The proof of the claim follows line by line from the one in~\cite[p. 25]{cao17-OTtheorem}.
We note here that if $\varphi_k$ has singularities along the analytic set $Z_k$ then $X\backslash Z_k=\overline{X}\backslash(D\cup Z_k)$ has a complete K\"ahler metric (cf.~\cite[Lemma 12.9]{Deamilly-book-analytic}), so we can apply the $L^2$ existence theorem for solutions to $\Bar{\partial}$-equation; cf.~\cite[Proposition 13.4, Remark 13.5]{Deamilly-book-analytic}, also~\cite[Appendix]{cao17-OTtheorem}. While, in~\cite{cao17-OTtheorem}, $X$ is assumed to be weakly pseudoconvex, so that $X\backslash Z$ is a complete K\"ahler manifold for an analytic subset $Z$. The latter property is still true in our setting. 

  \smallskip
  \textbf{Step 3.} Up to passing to a subsequence, we assume that $((1-b_m'\circ\Psi)\tilde u-\gamma_m)_m$ 
converges weakly to some $U\in L^2(X,K_X+L|_X)_{\omega_p,h_0}$. It suffices to verify that $U$ is holomorphic and satisfies~\eqref{eq: OT-inequ}. 
Since $\Bar{\partial}((1-b_m'\circ\Psi)\tilde u-\gamma_m)=(m/a_m)^{\frac{1}{2}}\beta_m$ converges to zero and $\Bar\partial$ is a closed operator it follows that $F$ is holomorphic. The inequality \eqref{eq: OT-inequ} follows immediately from the one \eqref{eq: OT2} (cf. \cite[page 26]{cao17-OTtheorem}).

It remains to check that $U|_{X_0\cap V_\alpha}=u$. We solve the $\Bar{\partial}$-equation on $V_\alpha\backslash D$: $\bar\partial w_m=\beta_m$ and
\[\int_{V_\alpha\backslash D}|w_m|^2_{\omega_{\rm P},h_{a_m}}\, {\rm d}V_{X,\omega_{\rm P}}\leq C\int_{V_\alpha\backslash D}|\beta_m|^2_{\omega_{\rm P},h_{a_m}}\, {\rm d}V_{X,\omega_{\rm P}}\leq C_1 \]
for some uniform constant $C_1>0$. Then \[U_m:=(1-b_m'\circ\Psi)\Tilde{u}-\gamma_m-(\frac{m}{a_m})^{\frac{1}{2}}w_m \]
is a holomorphic function on $V_\alpha\backslash D$, and $U_m\to U$ weakly. Since $U_m|_{X_0\cap V_\alpha}=u$ we have $U|_{X_0\cap V_\alpha}=u$. 
\end{proof}
\begin{theorem}\label{thm: variation-psh2}
    Theorems~\ref{thm_variation} and
    \ref{thm: variation-psh} hold in the case when $X$ is a complex manifold fibered over a complex manifold $Y$, with  fiber $X_y=\overline{X_y}\backslash D$. 
\end{theorem}
\begin{proof} The statement in Theorem~\ref{thm_variation} for the family $p:X=\overline{X}\setminus D\rightarrow Y$ follows immediately from the global regularization result and the Ohsawa--Takegoshi theorem above. 

    We then argue as in Theorem~\ref{thm: variation-psh} to obtain a sequence of singular Hermitian metrics $\phi_{k}$ on $K_{X/Y}^k|_{X^\circ}$ with positive curvature current $T_k:=dd^c \phi_k$.  
    Since $\phi_k$ uniformly converges to $\phi_{\rm KE}$ locally, we have $\phi_{\rm KE}$ is locally plurisubharmonic on $X^\circ$. Hence $\phi_{\rm KE}$ defines a singular Hermitian metric on $(K_{X/Y})|_{X^\circ}$, with positive curvature current {$T:= dd^c \phi_{\rm KE}$}.  
\end{proof}

\bibliographystyle{alpha}
	\bibliography{file}	
\end{document}